\documentclass[11pt]{amsart}
\usepackage{amsmath, amssymb}
\usepackage{upgreek}
\usepackage{xcolor, url}
\usepackage[all]{xy}
\usepackage{graphicx}
\usepackage{bm}
\usepackage{enumerate}
\theoremstyle{plain}
\newtheorem{theorem}{Theorem}[section]

\newtheorem{prop}[theorem]{Proposition}

\newtheorem{proposition}[theorem]{Proposition}
\newcommand{\td}{\text{d}}
\theoremstyle{definition}
\newtheorem{remark}[theorem]{Remark}
\newtheorem{definition}[theorem]{Definition}

\newtheorem{example}[theorem]{Example}
\numberwithin{equation}{section}

\DeclareMathOperator{\sgn}{sgn}

\usepackage{subfig, tikz}
\usetikzlibrary{decorations.pathmorphing}
\usepackage{graphics}

\usepackage{hyperref} 
\hypersetup{colorlinks=true, linkcolor=blue, citecolor = blue, urlcolor=blue, filecolor=magenta}

\allowdisplaybreaks

\begin{document}
\title[On the existence of toric ALE and ALF gravitational instantons]{On the existence of toric ALE and ALF gravitational instantons}
\author{Hari K. Kunduri}
\address[Hari K. Kunduri]{Department of Mathematics and Statistics and Department of Physics and Astronomy\\
		McMaster University\\
		Hamilton, ON Canada}
	\email{kundurih@mcmaster.ca}
    \author{James Lucietti}
\address[James Lucietti]{School of Mathematics and Maxwell Institute for Mathematical Sciences\\
		University of Edinburgh\\
		James Clerk Maxwell Building, Edinburgh, EH9 3FD, United Kingdom}
	\email{j.lucietti@ed.ac.uk}

\begin{abstract}
We establish existence and uniqueness results for asymptotically locally Euclidean (ALE) and asymptotically locally flat (ALF) gravitational instantons. In particular, we prove the existence of a unique, Ricci-flat, toric ALE and ALF gravitational instanton, for every  admissible rod structure, that is smooth up to possible conical singularites. We also give an elementary proof that any  toric ALE or ALF self-dual instanton is a multi-Eguchi-Hanson or multi-Taub-NUT solution.
\end{abstract}

\maketitle
\section{Introduction}

A gravitational instanton is a four-dimensional complete Riemannian manifold $(M, \mathbf{g})$ that is a solution to the Einstein equations with a prescribed curvature decay. These were introduced by Hawking in his study of Euclidean quantum gravity~\cite{Hawking:1976jb}. Two notable classes are the so-called asymptotically locally Euclidean (ALE) and asymptotically locally flat (ALF) instantons. ALE instantons have quartic volume growth and  approach $\mathbb{R}^4/\Gamma$ where $\Gamma$ is a finite subgroup of $O(4)$, whereas ALF instantons have cubic volume growth and a boundary at infinity of topology $S^1 \times S^2$ or $S^3/\Gamma$.  Most of the studies to date have concerned hyper-K\"ahler instantons, which have been  classified by the volume growth at infinity: ALE (quartic), ALF (cubic), ALG, ALG$^*$,  ALH  or ALH$^*$  (quadratic or slower)~\cite{Minerbe:2010yrr, ChenChen, SunZhang}. In particular, for  hyper-K\"ahler instantons, the ALE case is given by Kronheimer's construction on minimal resolutions of $\mathbb{C}^2/\Gamma$ for any $\Gamma\subset SU(2)$~\cite{Kronheimer:1989zs, Kronheimer:1989pu}, whereas Minerbe has proven that the ALF instantons have a boundary at infinity which is an $S^1$-bundle over $S^2$  (type $A_k$) or $\mathbb{RP}^2$ (type $D_k$) and  that ALF-$A_k$ instantons are exhausted by the multi-Taub-NUT solution~\cite{Minerbe, Minerbe:2009gj}.

Much less is known about generic  gravitational instantons. 
It is interesting to  consider the class of Ricci-flat asymptotically flat (AF) instantons, which corresponds to the subclass of ALF where the boundary at infinity is $S^1\times S^2$. Notably, it includes the Riemannian version of the Kerr solution, which by analogy to the black hole no-hair theorem had been conjectured to be the unique solution in this class~\cite{Gibbons, Lapedes:1980st}. Strikingly, Chen and Teo showed that this conjecture is false, by constructing an explicit counterexample now known as the Chen-Teo instanton~\cite{Chen:2011tc}. It turns out that, like the Kerr metric, the Chen-Teo metric possesses a toric symmetry. In this case the notion of {\it rod structure}  can be defined for such instantons~\cite{Chen:2010zu}, which is data that encodes how the torus action degenerates on the axes and determines the topology of the manifold. In particular, the Kerr and Chen-Teo instantons possess different rod structures, so this provides data to distinguish these solutions. 

In a previous paper, we initiated the classification of generic, Ricci-flat, toric AF instantons~\cite{Kunduri:2021xiv}.  We showed that one can formulate the Einstein equations as a harmonic map and use this to prove that there exists a unique toric AF instanton for any admissible rod structure, although it may still suffer from conical singularities. While the harmonic map reduction of the Einstein equations is well-known in the context of stationary and axisymmetric spacetimes, we stress that for toric gravitational instantons a different (more direct) reduction  is required to prove these results. Therefore, this work left open the interesting question of what rod structures lead to smooth instantons. In a remarkable recent paper, Li and Sun have used this harmonic map formulation to prove the existence of an infinite class of new AF instantons that are free of conical singularities and hence everywhere smooth~\cite{LiSun}. Nevertherless, the complete classification of toric AF instantons remains a notable open problem.

Curiously, the Kerr and Chen-Teo instantons also possess special geometry: they are Hermitian non-K\"ahler~\cite{Aksteiner:2021fae}. Motivated by this, Biquard and Gauduchon derived an explicit classification of  toric ALF Hermitian non-K\"ahler instantons~\cite{Biquard:2021gwj}. This revealed a finite list: Kerr, Chen-Teo, Taub-NUT and Taub-Bolt. Furthermore, Li has shown that an ALF Hermitian non-K\"ahler instanton must be toric, therefore completing their classification~\cite{LiALF}. It has been conjectured that any ALF instanton must be Hermitian (K\"ahler or non-K\"ahler)~\cite{Aksteiner:2023djq}. However, the aforementioned Li-Sun instantons are necessarily non-Hermitian and hence these provide counterexamples to this conjecture for the AF subclass. Nevertheless, the conjecture remains open in the generic ALF class.  On the other hand, there is an older more stringent conjecture which states that  any  Ricci-flat ALE instanton must be (locally) hyper-K\"ahler~\cite{Gibbons, BKN}. Support for this conjecture has been obtained under a variety of topological and geometrical assumptions: for spin manifolds with $\Gamma\subset SU(2)$~\cite{Nakajima}, for Hermitian non-K\"ahler manifolds with $\Gamma \subset SU(2)$~\cite{Li2023}, and for toric Hermitian non-K\"ahler instantons~~\cite{Araneda:2025uqo}. 

Motivated by these conjectures, we will initiate the classification of generic toric ALE and ALF gravitational instantons.  In particular, the purpose of this note is to establish an existence and uniqueness result analogous to the AF case~\cite{Kunduri:2021xiv}. Our main result is as follows.

\begin{theorem}
\label{th:exist}
There exists a unique toric ALE or ALF gravitational instanton for every admissible rod structure, that is smooth everywhere, up to possible conical singularities at the axes.
\end{theorem}

The method of proof is entirely analogous to the AF case. In particular, we use the harmonic map formulation of the Einstein equations and a theorem of Weinstein which reduces the existence problem for the harmonic map to the construction of a `nearby' model map that shares the same boundary conditions~\cite{Weinstein:2019zrh}. It turns out that the construction of a suitably general ALE or ALF  model map is more delicate than in the AF case and most of this note is dedicated to constructing this explicitly.\footnote{For a restricted class of ALE and ALF asymptotics, namely those with an $S^3$ boundary at infinity, this existence result was asserted without proof in~\cite{LiSun}. However, in the toric setting the boundary at infinity is a general lens space.} Of course, as in the AF case, our theorem does not deal with the conical singularities and it is possible that all non-Hermitian examples are necessarily conically singular, in line with the above conjectures. In fact, topological constraints on the possible rod structure for ALE and ALF instantons have been derived~\cite{Nilsson:2023ina}. On the other hand, in light of the Li-Sun construction of AF instantons, it is possible that non-Hermitian instantons may also exist in the ALE or ALF case. We leave this interesting question to future work.

It is also  worth noting that, in contrast to the hyper-K\"ahler case, for generic gravitational instantons there is no classification of asymptotic types even in the toric case. Indeed, it is known there are toric instantons that are asymptotically Kasner such as the Myers-Korotin-Nicolai solutions~\cite{Khuri:2022xdy}. It would be interesting to determine all possible asymptotic types for Ricci-flat toric instantons and establish the analogue of the above theorem for each type.

We also develop an asymptotic expansion for toric ALF instantons. In particular, we determine the explicit leading order asymptotics and show that it is specified by three asymptotic invariants that may be identified as Riemannian analogues of a mass, NUT-charge and angular momentum.  Indeed, in the case of vanishing NUT-charge this reduces to the  mass and angular momentum previously derived for AF instantons~\cite{Kunduri:2021xiv}.  While these asymptotic expansions are not needed to establish the uniqueness and existence theorems, it may be interesting to study ALF generalisations of the mass and angular momentum identities established for AF instantons~\cite{Kunduri:2021xiv}.

Finally, for completeness, we give an elementary proof that shows that any self-dual toric ALE or ALF  instanton must be the multi-Eguchi-Hanson or  multi-Taub-NUT instanton.  This follows from the well-known fact that a self-dual toric instanton is locally a Gibbons-Hawking metric, together with a global analysis of this family of metrics that exploits the rod structure. Of course, these results must follow from the aforementioned general classification of hyper-K\"ahler ALE and ALF restricted to the toric case, however, it is instructive to analyse these directly in the toric setting as they provide a large class of examples where the conical singularities can be explicitly removed.

The organisation of this paper is as follows. In Section \ref{sec2} we introduce toric gravitational instantons  and develop an asymptotic model for ALE and ALF asymptotics. In Section \ref{sec3} we introduce the harmonic map formulation and prove our main theorem. In Section \ref{sec:SD} we show how to determine all self-dual toric ALE and ALF instantons. \\

\paragraph{\bf Acknowledgments} We thank the Simons Centre for Geometry and Physics for hospitality where part of this work was conducted.  We thank Bernardo Araneda and Marcus Khuri for many helpful discussions and comments on a draft of this work. H. K. is supported by NSERC Discovery Grant RGPIN-2025-06027.

\section{Toric ALE and ALF gravitational instantons}
\label{sec2}

\subsection{Einstein equations and rod structure}  We begin by introducing the  basic definitions and properties of toric instantons,  as discussed in more detail in~\cite{Kunduri:2021xiv}. We then introduce the asymptotic types that we will consider. 
\begin{definition}
    A {\it toric gravitational instanton} is a $4$-dimensional, complete, simply connected, Riemannian manifold $(M,\mathbf{g})$ that is Ricci flat, which admits an (effective) isometric $2$-torus $T$ action with at least one fixed point. We assume this action has no points with discrete isotropy subgroups.
\end{definition}

It has been shown that under these assumptions the orbit space $\hat M =M/T$ is a 2-dimensional simply connected manifold with boundaries and corners~\cite[Proposition 1]{Hollands:2007aj} (see also the classic work~\cite[Theorem 1.12]{OR}).  The Gram matrix of the Killing fields $\eta_i$, $i=1,2$ generating the $T$-action is the $2\times 2$ symmetric matrix $g_{ij}:= \mathbf{g}(\eta_i, \eta_j)$. The interior, boundaries, and corners, of $\hat{M}$ correspond to points where $g$ is rank-2, rank-1 and rank-0 respectively.  It is a standard result~\cite{Kunduri:2021xiv} that Ricci flatness of $(M, \mathbf{g})$ implies that the $T$-isometry with a fixed point is orthogonally transitive and hence the orthogonal surfaces, which we may identify with the orbit space, possess an induced metric $\hat{g}$.  In particular, Ricci-flatness implies that the Gram matrix viewed as a function on $(\hat M, \hat g)$, satisfies
\begin{equation}\label{eq:geq}
    \td (\rho \hat\star g^{-1} \td g)=0
\end{equation}
where $\rho:= \sqrt{\det g}$ and $\hat \star$ is the Hodge dual with respect to the metric $\hat{g}$.  Taking the trace of this equation implies that $\rho$ is harmonic and hence we can introduce its harmonic conjugate by $\td z = - \hat\star \td \rho$. Observe that since $\hat M$ is simply connected $z$ is a globally defined function on the orbit space. Thus wherever $\td \rho \neq 0$ we can use $(\rho, z)$ as local coordinates on the orbit space and write $\hat g = e^{2\nu} (\td \rho^2+ \td z^2)$ for some function $\nu$. These coordinates can be pulled-back to $M$ to give the well-known Weyl-Papapetrou coordinates $(\rho, z, \phi^i)$ with $i=1,2$, 
\begin{equation}
\mathbf{g}= g_{ij} \td \phi^i \td \phi^j +  e^{2\nu} (\td \rho^2+ \td z^2)
\end{equation}
where $\eta_i = \partial_{\phi^i}$. The remaining components of the Einstein equations imply
\begin{equation}\label{eq:nuint}
\partial_z \nu = \frac{\rho}{4} \text{Tr}( J_z J_\rho), \qquad \partial_\rho \nu = - \frac{1}{2\rho}+ \frac{\rho}{8} \text{Tr}( J_\rho^2- J_z^2)
\end{equation}
where $J:= g^{-1} \td g$, and the integrability condition of these is precisely \eqref{eq:geq}.

In fact, for the asymptotics that we will consider we will show that the Weyl-Papapetrou coordinates is a global chart on $M$ away from the axes (i.e. at points corresponding to the interior of $\hat{M}$).  We first introduce the asymptotics of interest in a general setting.

\begin{definition}
A Riemannian manifold $(M, \mathbf{g})$ is ALE  of order $\tau>0$ if~\cite{Araneda:2025uqo}:
\begin{itemize}
    \item It has an end  diffeomorphic to $\mathbb{R}\times S$ where $S= S^3/\Gamma$ and $\Gamma$ is a finite subgroup of $O(4)$.
    \item $S$ admits a Riemannian metric $\gamma$ of constant curvature $+1$.
    \item On the end
    \begin{equation}
        \mathbf{g} = \td r^2+ r^2 \gamma + h
    \end{equation}
    where $h= O_k(r^{-\tau})$, which means that  $|\nabla^i h| = O(r^{-\tau-i})$ as $r \to \infty$ for all integers $0\leq i
    \leq k$  and the norm $|\cdot|$ and connection $\nabla$ are with respect to the asymptotic flat metric $\td r^2+ r^2 \gamma$. It is worth noting that Ricci-flatness implies $\tau=4$~\cite{BKN}.
\end{itemize}
\end{definition}

\begin{definition}\label{def:ALF}
A Riemannian manifold $(M, \mathbf{g})$ is ALF if~\cite{Biquard:2021gwj}:
\begin{itemize}
    \item It has an end  diffeomorphic to $\mathbb{R}\times S$ where $S$ is either $S^1\times S^2$ or $S^3/\Gamma$ and $\Gamma$ is a finite subgroup of $O(4)$.
    \item $S$ admits a 1-form $\eta$, a vector field $\xi$ such that $\eta(\xi)=1$, $\iota_\xi \td \eta=0$, and a $\xi$-invariant  metric $\gamma$ on $\text{ker}\,\eta$ of constant curvature $+1$ extended to $TS$ so $\gamma(\xi, \cdot)=0$.
    \item On the end
    \begin{equation}
        \mathbf{g} = \td r^2+ r^2 \gamma + \eta^2+ h
    \end{equation}
    where $h= O_k(r^{-1})$, which means that  $|\nabla^i h| = O(r^{-1-i})$ as $r \to \infty$ for all integers $0\leq i
    \leq k$  and the norm $|\cdot|$ and connection $\nabla$ are with respect to
    the asymptotic  metric $\td r^2+ r^2 \gamma+\eta^2$.  These conditions imply $\xi$ is a Killing field of this asymptotic metric.
\end{itemize}
\end{definition}

We must now define what it means for a toric instanton to be one of these asymptotic types, which requires a compatibility condition of the asymptotic structures with the toric symmetry.

\begin{definition}
A Riemannian manifold $(M, \mathbf{g})$ that is {\it toric} (i.e. admits an effective torus $T$ isometric action) is:
\begin{itemize}
    \item Toric ALE if the $T$-action is an action on $S$ that preserves $\gamma$.
    \item Toric ALF if the $T$-action is an action on $S$ that preserves $(\eta, \xi , \gamma)$ and $\xi$ generates a subgroup of the $T$-action.
\end{itemize}
The torus action implies $\Gamma$ must be a cyclic group so in the spherical case $S$ is a lens space $L(p,q)$.
\end{definition}

In order to make use of the harmonic map formulation of the Einstein equations, we need to establish that Weyl coordinates are global in our context. The proof follows an old argument of Weinstein for stationary and axisymmetric black holes~\cite{wein}.

\begin{proposition}
    Let $(M,\mathbf{g})$ be a toric ALE or ALF gravitational instanton. Then, the Weyl-Papapetrou coordinates are a global chart on the interior of the orbit space $\hat M$.
\end{proposition}

\begin{proof}
    This has been already proven in the AF case~\cite{Kunduri:2021xiv} and the ALE case~\cite{Araneda:2025uqo}. The proof for the ALF case goes through with minimal changes, but we give it here for completeness.  The definition of toric ALF implies that the orbit space has an asymptotic end $\mathbb{R}\times S/T$ where the orbit space $S/T$ is a closed interval and the matrix of Killing fields of the asymptotic metric $\sigma_r:= \eta^2+ r^2 \gamma$ induced on $S$ is rank-1 at the endpoints.  Thus the asymptotic end of the orbit space is a strip, which is simply connected and the Gram matrix with respect to the Killing fields $\xi, K$, where $K\in \text{ker}\, \eta,$ is $\text{diag}( 1, r^2 \gamma(K, K))$. Hence,  $\rho = c r \sqrt{\gamma(K, K)} (1+ O(r^{-1}))$ for some constant $c$ related to a change of basis of the torus Killing fields. Therefore, for large enough $\rho_0>0$, the curve $\rho=\rho_0$ is in the asymptotic end, which implies the region $\rho>\rho_0$ is simply connected.  But since $\hat M$ is simply connected, this implies the complement $0<\rho <\rho_0$ is also simply connected. Hence, by the uniformisation theorem we can conformally map the region $0< \rho <\rho_0$ to the strip $0< \text{Im}\, w < \rho_0$ in the complex $w$-plane such that $\rho=0$ corresponds to $\text{Re} \, w=0$ and $\rho=\rho_0$ to $\text{Im} \, w = \rho_0$. Therefore, $\rho - \text{Im}\, w $ is harmonic in this strip of the $w$-plane and vanishes at its boundaries, so by the maximum principle $\rho = \text{Im} \, w$. Therefore, $w= z+ i \rho$ is a global holomorphic coordinate on the orbit space. 
\end{proof}

This proposition is key as it means the interior of the orbit space is diffeomorphic to the half-plane $H= \{ (\rho, z) \; | \; \rho>0\}$. As shown in~\cite{Hollands:2008fm}, its boundary $\rho=0$ divides into intervals, called rods, $I_i=(z_{i-1}, z_{i})$ separated by points $z_i, i=1, \dots, n$ where $z_0=-\infty, z_{n+1}=\infty$. At the rods $I_i$, $i= 1, \cdots, n+1$  the Gram matrix $g$ is rank-1 with rod vector ${v}_i\in \text{ker}\, g|_{I_i}$, whereas the points $z_i$ correspond to fixed points of the $T$-action.  The {\it rod structure} of the instanton is the collection of data $\{ (I_i, {v}_i)\}$ and we normalise the rod vectors so they are $2\pi$-periodic.  It is often convenient to use a $2\pi$-periodic basis for the generators for the $T$-action, relative to which the rod vectors can be written as $\underline{v}_i= (v^1_i, v^2_i)\in \mathbb{Z}^2$ (coprime).  The rod structure is then said to be {\it admissible} if 
\begin{equation}
    \det (\underline{v}_i, \underline{v}_{i+1})=\epsilon_i, \qquad \epsilon_i=\pm 1
\end{equation} for all $i=1, \dots, n+1$, which ensures the absence of orbifold singularities at the fixed points $z_i$~\cite{Hollands:2008fm}. 

Given any rod $I_i$, we can introduce independently $2\pi$-periodic torus coordinates $(\phi^1, \phi^2)$ adapted to the corresponding rod vector so $v_i = \partial_{\phi^2}$ such that the Gram matrix in this basis is
\begin{equation}\label{eq:gparam}
g= \begin{pmatrix}
    h+ \rho^2 h^{-1} w^2 & \rho^2 h^{-1} w \\ \rho^2 h^{-1} w & \rho^2 h^{-1}
\end{pmatrix}
\end{equation}
where $h>0$ and $w$ are smooth functions of $(\rho^2, z)$, so in particular we have $|v_i|^2 = \rho^2 h^{-1}$. Integrating \eqref{eq:nuint} near $I_i $ implies that~\cite{Kunduri:2021xiv}
\begin{equation}
e^{2\nu} = \frac{c_i^2}{h}
+ O(\rho^2) \; ,   \label{eq:nunearI}
\end{equation}
where  $c_i>0$ is a constant, so the metric near $I_i$ is 
\[
\mathbf{g} = h^{-1}[  ( c_i^2+ O(\rho^2)) \td \rho^2 + \rho^2 ( \td \phi^2+ w \td \phi^1)^2 ]+ \left( \frac{c_i^2}{h} + O(\rho^2) \right) \td z^2 + h (\td \phi^1)^2 \; .
\]
Thus there is a conical singularity as $\rho \to 0$ over $I_i$ with cone angle $\theta_i=2\pi/c_i$,  which is absent iff $c_i=1$. The finite rods $I_i$ correspond to closed 2-surfaces $C_i$ of $S^2$ topology in $M$. The induced metric on $C_i$ is
\begin{equation}
\frac{c_i^2 \td z^2}{h(z)}+ h(z) (\td \phi^1)^2   \label{eq:bolt}
\end{equation}
and since $ |\partial_{\phi^1}|^2=h$ we must have $h>0$ on the interior of $I_i$ with $h=0$ at its endpoints (since they correspond to fixed points of the torus symmetry). Thus the metric on $C_i$ extends smoothly to $S^2$ if and only if
\begin{equation}
h'(z_{i-1})=- h'(z_{i})=2 c_i  \; ,  \label{eq:smoothbolt}
\end{equation}
which is the condition for the absence of conical singularities at the endpoints of $I_i$.

\subsection{Toric ALF instantons}

It will be useful to introduce an explicit set of coordinates adapted to the asymptotic types introduced. In the ALE case this is straightforward and will be postponed to the next section.  In the ALF case this requires unpacking its general definition. The following result can also be found in~\cite{Biquard:2021gwj} (see also~\cite{Aksteiner:2021fae}). 

\begin{proposition}\label{prop:ALFend}
Let $(M, \mathbf{g})$ be an ALF instanton. There exist coordinates $(r, \theta, \phi, \tau)$ on the asymptotic end such that the metric takes the form
\[
\mathbf{g}= \td \tau^2+ \td r^2+ r^2 (\td\theta^2+  \sin^2\theta \td \phi^2 ) +O_k(r^{-1}) ,
\]
as $r \to \infty$.  If $(M, \mathbf{g})$ is toric ALF  then $\xi=\partial_\tau$ is a Killing vector field and we may choose coordinates on $S$ so that $K=\partial_\phi$ is also a Killing vector. 
\end{proposition}

\begin{proof}
First note that by Definition \ref{def:ALF} ALF implies $\mathcal{L}_\xi\eta=0$.  Now let $(\tau, \theta, \phi)$ be coordinates on $S$ adapted so $\xi=\partial_\tau$. It follows from $\eta(\xi)=1$ that in these coordinates $\eta=\td \tau + \omega$ where $\iota_\xi\omega=0$ and $\mathcal{L}_\xi\omega=0$, so we can write $\omega = \omega_\theta(\theta, \phi) \td \theta+ \omega_\phi(\theta, \phi)\td \phi$.  Now consider the symmetric 2-tensor $\gamma$ on $S$ which by definition satisfies $\gamma(\xi, \cdot )=0$ and is also invariant under $\xi$. In these coordinates $\gamma= \gamma_{\theta\theta} \td \theta^2+ 2 \gamma_{\theta \phi} \td \theta \td \phi+ \gamma_{\phi\phi} \td \phi^2$ where all components depend on $(\theta, \phi)$, so locally it takes the form of a metric on a 2d space with coordinates $(\theta, \phi)$. It thus gives a metric on $\text{ker} \, \eta$ and the requirements it has unit curvature implies it is locally isometric to the unit round metric on $S^2$. Thus we can always choose the coordinates so $\gamma= \td \theta^2+\sin^2\theta \td\phi^2$.  The asymptotic form of the metric is thus
\[
\mathbf{g}= \td r^2+ r^2( \td \theta^2+\sin^2\theta \td\phi^2)+ (\td \tau+\omega_\theta(\theta, \phi) \td \theta+ \omega_\phi(\theta, \phi)\td \phi )^2 +O_k(r^{-1}) \;
\]
Recall that here a tensor is $O(r^{-1})$ if its norm with respect to the asymptotic metric $\td r^2+r^2 \gamma +\eta^2$ is $O(r^{-1})$, and similarly $O_k(r^{-1})$ refers to the norm of its first $k$ covariant derivatives, see Definition \ref{def:ALF}.  However, since $\omega_\theta, \omega_\phi$ and their $\theta, \phi$ derivatives are $O(1)$, the 1-form  $\omega= O_k(r^{-1})$ and hence the claim follows. 
\end{proof}

While the above proposition is sufficient for establishing the uniqueness and existence theorem, it is interesting to develop a more detailed asymptotic expansion.  We first deduce the asymptotic form of the Weyl-Papapetrou coordinates. Indeed, Proposition \ref{prop:ALFend} implies that with respect to the Killing fields $\partial_\tau, \partial_\phi$ we have
\begin{equation}
\rho = r \sin\theta + O(1), \qquad z = r \cos\theta + O(\log r)
\end{equation}
as $r \to \infty$.
Then, by writing the metric in Weyl-Papapetrou coordinates we can improve the asymptotic expansion of the metric as follows. To this end, it is useful to introduce polar coordinates in the $(\rho, z)$ plane by $\rho = R \sin \Theta$ and $z= R \cos\Theta$.  It follows that $R= r +O(\log r)$ and hence the asymptotic end $r\to \infty$ corresponds to $R\to \infty$.   

\begin{proposition}\label{prop:ALFweyl}
    Let $(M, \mathbf{g})$ be a toric ALF instanton. Then in Weyl-Papapetrou polar coordinates the metric can be written as
    \begin{equation}\label{eq:Weylpolar}
    \mathbf{g}=  V ( \td \tau+ \omega \td \phi)^2 + V^{-1} R^2 \sin^2\Theta \td \phi^2 + e^{2\nu} (\td R^2+ R^2 \td \Theta^2)
    \end{equation}
    where 
    \begin{equation}
V = 1- \frac{2m}{R}+ O_k(R^{-2}), \quad \omega = N \cos\Theta+ \frac{2j \sin^2\Theta}{R} + O_k(R^{-2}\log^2 R)  , \quad e^{2\nu} = 1 + \frac{2m}{R} + O_k(R^{-2}),
    \end{equation}
where $N, m$ and $j$ are constants.
\end{proposition}

\begin{proof} We can parameterise the Gram matrix in the basis $\partial_\tau, \partial_\phi$ of Proposition \ref{prop:ALFend}  by functions $V$ and $\omega$ as above.  Then 
the twist potential $Y$ with respect to $\partial_\tau$ satisfies~\cite{Kunduri:2021xiv} 
 \begin{equation}
 \td Y= - \frac{V^2}{\rho} \star_2 \td \omega  \label{Wdef}
 \end{equation}
The equation (\ref{eq:geq}) in this parameterisation reduces to
 \begin{align}
 \Delta_3 \log V &= V^{-2} \nabla Y \cdot \nabla Y  \label{Veq} \\
 \Delta_3 Y &= 2\nabla \log V \cdot \nabla Y  \label{Weq}
 \end{align}
 where $\nabla, \Delta_3$ are the $\mathbb{R}^3$ derivative and Laplacian  respectively, on an auxiliary $\mathbb{R}^3$ with cylindrical coordinates $(\rho, z, \varphi)$ and $V, W$ are axially symmetric.  
 
 Now comparing to Proposition \ref{prop:ALFend} and noting that $R= r+O(\log r)$ we deduce that as $R\to \infty$
 \begin{equation}
 V=1+ O_k(R^{-1}), \qquad \omega= O_k(1), \qquad \nu = O_k(R^{-1})
 \end{equation}
 where $f=O_k(R^{-1})$ if $\partial^i f=O(R^{-1-i})$ for $i\leq k$ in cartesian coordinates on the $\mathbb{R}^3$.  We now determine the explicit form of the leading terms in the asymptotic expansion,  using the results of Beig and Simon who determined the expansion for asymptotically flat stationary spacetimes~\cite{BeigI, BeigII}.
 
 Integrating for the twist potential we find  $Y= O_k(R^{-1})$.  Therefore, the r.h.s. of \eqref{Weq} is $O_{k-1}(R^{-4})$ so by \cite[Lemma B]{BeigI} there is a constant $N$ such that
 \begin{equation}
 Y= \frac{N}{R}+O_k(R^{-2}\log R) \; .
 \end{equation}
 Then (\ref{Veq}) implies
 \begin{equation}
 \Delta_3 \left( \log V- \frac{N^2}{2 R^2} \right)= O_{k-1}(R^{-5}\log R)
 \end{equation}
 where we have used $\Delta_3(R^{-2})= 2 R^{-4}$. By \cite[Lemma A]{BeigII}, it follows that there exist constants $m, m_1$ such that
 \begin{equation}
 \log V= -\frac{2m}{R}+\frac{N^2}{2 R^2} +\frac{m_1 \cos\Theta}{R^2} +O_{k}(R^{-3}\log^2 R)
 \end{equation}
 and hence in particular
 \begin{equation}
 V= 1-\frac{2m}{R}+O_k(R^{-2}) \; .
 \end{equation}
 Then (\ref{Weq}) implies
 \begin{equation}
 \Delta_3 \left( Y+\frac{2mN}{R^2} \right) =O_{k-1}(R^{-5}\log R)
 \end{equation}
 and hence by \cite[Lemma A]{BeigII} again there exist a constant $j$ such that
 \begin{equation}
 Y=\frac{N}{R}-\frac{2m N}{R^2}-\frac{2j \cos\Theta}{R^2}+O_k(R^{-3}\log^2 R)
 \end{equation}
 Integrating (\ref{Wdef}) then implies
 \begin{equation}
 \omega= N \cos \Theta+ \frac{2 j \sin^2\theta}{R}+O_k(R^{-2}\log^2 R)  \; .
 \end{equation}
Finally, \eqref{eq:nuint} can be rewritten in terms of the $(R,\Theta)$ coordinates as
\begin{equation}
    \partial_R \nu = -\frac{m}{R^2} + O_{k-1}(R^{-3}), \qquad \partial_\Theta \nu = O_{k-1}(R^{-2}).
\end{equation} Integrating for the conformal factor leads to the claimed result. \end{proof}

\begin{example}[Taub-NUT]\label{ex:TN}
The self-dual Taub-NUT metric in Gibbons-Hawking form is given by 
\begin{equation}
    \mathbf{g}_{\text{TN}} = H^{-1} (\td \tau+ N \cos\theta \td \phi)^2+ H ( \td r^2+ r^2 \td \theta^2+ r^2 \sin^2 \theta \td 
    \phi^2), 
\qquad H= 1+\frac{N}{r}
\end{equation}
where $N$ is a constant.  If $N>0$ this metric extends smoothly  to $r=0$ which is a fixed point of the torus symmetry generated by $(\partial_\tau, \partial_\phi)$ and gives a smooth metric on $M= \mathbb{R}^4$. This metric is already expressed in Weyl polar coordinates and hence comparing to Proposition \ref{prop:ALFweyl} we  immediately deduce that the metric is ALF with $m= \tfrac{1}{2} N$ and $j=0$.
\end{example}

\begin{example}[Taub-Bolt] The Taub-Bolt instanton is on $M=\mathcal{O}(-1)$ and is ALF with $S^3$ topology at infinity. The metric is given by
 \begin{equation}\begin{aligned}
 \mathbf{g}_{\text TB} &= U(r) \left(\td \tau+N \cos\theta \td \phi\right)^2 +  \frac{\td r^2}{U(r)} + \left(r^2 - \frac{N^2}{4} \right) (\td\theta^2 + \sin^2\theta \td \phi^2) \\
 U(r) & = \frac{(r - N)(4r  -N)}{4r^2 -N^2}
 \end{aligned}
 \end{equation} where $N>0$, $r \in (N, \infty), \theta \in (0,\pi)$. The torus symmetry generated by $(\partial_\tau, \partial_\phi)$ is such that $\tau = N(2 \phi^1 +  \phi^2), \phi = \phi^2$ where $(\phi^1, \phi^2)$ are independently $2\pi$ periodic coordinates, which ensures the metric extends smoothly to the axes $\theta=0,\pi$ and a bolt at $r=N$. The rod structure is given by rods $(r>N, \theta=\pi)$, $(r=N, 0<\theta<\pi)$, $(r>N, \theta=0)$ with rod vectors $(0,1)$, $(1,0)$, $(1,-1)$ respectively (in the basis $(\partial_{\phi^1}, 
 \partial_{\phi^2})$). Weyl-Papapetrou coordinates with respect to $(\partial_\tau, \partial_\phi)$ are
 \begin{equation}
     \rho =  \frac{1}{2} \sqrt{(r - N)(4r - N)}\sin\theta, \qquad z = \left( r- \frac{5N}{8}\right) \cos\theta 
 \end{equation} 
 and converting to Weyl polar coordinates one finds in particular  $r = R + O(1)$. Comparing to Proposition \ref{prop:ALFweyl}, a short computation confirms the metric is ALF with $m= \tfrac{5}{8}N$ and $j=0$.
 \end{example}
    
   If $N=0$ in the previous  proposition the metric is AF. Indeed, if we let $\phi^i$ be $2\pi$-periodic torus coordinates, then the corresponding Killing fields $\partial_{\phi^i}$ are related to $\partial_\tau, \partial_\phi$ by a $GL(2, \mathbb{R})$ transformation of the form $\partial_{\phi^2}= \hat\beta (\partial_\tau+ \Omega \partial_\phi)$ and $\partial_{\phi^1}= \partial_\phi$ for some constants $\hat{\beta}>0$ and $\Omega$. In terms of the coordinates this reads $\tau= \hat\beta \phi^2$ and $\phi = \phi^1+ \hat\beta \Omega \phi^2$, which implies the identifications $(\tau, \phi)\sim (\tau+ \beta, \phi+ \beta \Omega)$ and $(\tau, \phi)\sim (\tau, \phi+2\pi)$.  Therefore, from Proposition \ref{prop:ALFweyl} we deduce that  the definition of AF used in~\cite{Kunduri:2021xiv} follows from the more general definition of ALF used in this paper. Furthermore, we recover the asymptotic expansion obtained there for the AF case, where $m$ and $j$ were identified as Riemannian analogues of the mass and angular momentum.  
   
      Similarly, Proposition \ref{prop:ALFweyl} shows that in the ALF case we have  asymptotic invariants $N, m, j$ which are natural to identify as Riemannian analogues of the `NUT charge', the mass and the angular momentum.  Minerbe~\cite{Minerbe} has defined a mass for ALF manifolds (not necessarily toric or Ricci flat), although he assumed that the circle at infinity has bounded length (so in the AF case $\Omega=0$). In contrast we have obtained a mass for toric ALF instantons where we do not assume the circle at infinity is bounded.  It would be interesting to define a more general mass for ALF manifolds that unifies these two definitions.

\subsection{Toric ALE instantons}

We now introduce coordinates adapted to the ALE case. This is simpler since the  cross-section at infinity $(S, \gamma)$ is given by $S=L(p,q)$ with $\gamma$ maximally symmetric with unit curvature, so it must be locally isometric to the unit round metric on $S^3$. Therefore, the metric at infinity $\td r^2+ r^2 \gamma$ is a flat cone. There are of course many coordinate systems adapted to the toric symmetry one could use, but perhaps one of the simplest are the Hopf coordinates which are introduced as follows. A lens space $L(p,q)$ may be defined as a quotient of $S^3=\{ (z_1, z_2) \in \mathbb{C}^2 \; | \; |z_1|^2+ |z_2|^2=1\}$ under the $\mathbb{Z}_p$-action $(z_1, z_2)\sim (e^{2\pi i/p} z_1, e^{2\pi i q/p}z_2)$ where $p\in \mathbb{N}$ and $q$ is defined mod $p$ so we take $q=0, 1, \dots, p-1$. Introducing coordinates $z_i= r_i e^{i \bar{\phi}_i}$,  $r_1= \bar r \sin \bar\theta$, $r_2= \bar r \cos\bar\theta$ with $0\leq \bar\theta \leq \pi/2$,  the flat Euclidean metric $|\td z_1|^2+ |\td z_2|^2$ on $\mathbb{C}^2$ induces a flat cone metric with cross-section $S=L(p,q)$, 
\begin{equation}
{\mathbf{g}_0}= \td \bar r^2+ \bar r^2 ( \td \bar\theta^2+ \sin^2\bar\theta \td \bar{\phi}_1^2+ \cos^2\bar\theta \td \bar{\phi}_2^2 )   \label{eq:ALEflat}
\end{equation}
where  
\begin{equation}\label{eq:lens}
       (\bar{\phi}_1,\bar{\phi}_2) \sim \left(\bar{\phi}_1+ \frac{2\pi}{p}, \bar{\phi}_2+ \frac{2\pi q}{p} \right)  \qquad (\bar{\phi}_1,\bar{\phi}_2) \sim (\bar{\phi}_1, \bar{\phi}_2+2\pi) \;.
\end{equation} This gives coordinates $(\bar\theta, \bar{\phi}_1, \bar{\phi}_2)$ on $S$ with the above identifications and $\gamma$ is easily read off.   Therefore, in these coordinates the metric for an ALE instanton near infinity  takes the form
\begin{equation}\label{eq:ALEbar}
    \mathbf{g}= \mathbf{g}_0+ O(\bar{r}^{-\tau})
\end{equation}
as $\bar r \to\infty$.
It will be convenient to also use Euler angles $(\theta, \psi, \phi)$ on $S$, defined by $\bar\theta= \theta/2$ and 
\begin{equation}
\label{eq:euler}
\bar{\phi}^1= \frac{\phi-\psi}{2}, \qquad \bar{\phi}^2= \frac{\phi+\psi}{2}
\end{equation}
in terms of which the asymptotic metric is 
\begin{equation}
{\mathbf{g}_0}= \td \bar r^2+ \tfrac{1}{4} \bar r^2 ( \td \theta^2+ \sin^2\theta \td \phi^2+ (\td \psi+ \cos\theta \td \phi)^2)   \label{eq:ALEflat2}
\end{equation}
We can now deduce the asymptotic form of the metric in Weyl-Papapetrou coordinates. With respect to the basis of Killing fields $\partial_\psi, \partial_\phi$  we find 
\[
\rho =\tfrac{1}{4} \bar r^2 \sin\theta  +O(\bar r^{2-\tau}), \qquad z= \tfrac{1}{4} \bar r^2 \cos\theta  +O(\bar r^{2-\tau})
\]
and therefore the Weyl polar coordinate $R=\sqrt{\rho^2 + z^2}$ is $R= \tfrac{1}{4} \bar r^2  + O(\bar r^{2-\tau})$. Therefore the asymptotic end corresponds to $R\to \infty$ and inverting we deduce that $\tfrac{1}{4} \bar r^2= R+ O(R^{1-\tau/2})$.

It is useful to write the ALE and ALF $N\neq 0$ in Weyl-Papapetrou coordinates in a unified form.

\begin{proposition}\label{prop:ALEFWeyl}
    Let $(M, \mathbf{g})$ be a toric ALE or ALF $N\neq 0$ instanton. Then in Weyl-Papapetrou polar coordinates the metric can be written as
    \begin{equation}\label{eq:Weylpolar2}
    \mathbf{g}=  W^{-1} ( \td \psi+ \Omega \td \phi)^2 + W R^2 \sin^2\Theta \td \phi^2 + e^{2\nu} (\td R^2+ R^2 \td \Theta^2)
    \end{equation}
    where 
    \begin{equation}
    W= \begin{cases}  
     \frac{1}{R}+ O_k(R^{-1-\tau/2}) \\ \frac{1}{N^2} + O_k(R^{-1})
    \end{cases} 
    \quad
    \Omega =\begin{cases}
        \cos\Theta + O_k(R^{-\tau/2}) \\
        \cos\Theta + O_k(R^{-1})
    \end{cases}
    \quad e^{2\nu} =\begin{cases} \frac{1}{R}+O_k(R^{-1-\tau/2})\\  \frac{1}{N^2} + O_k(R^{-1})
    \end{cases}
    \quad \begin{array}{c}
        \text{ALE} \\ \text{ALF}
    \end{array}
\end{equation}
In both cases, the angles $(\psi, \phi)$ are subject to the identifications in  \eqref{eq:lens} and \eqref{eq:euler}, so the constant $R$ surfaces at infinity have lens space $L(p,q)$ topology.
\end{proposition}

\begin{proof}
First consider the ALE case. Relative to the basis $\partial_\psi, \partial_\phi$ we can parameterise the Gram matrix as in \eqref{eq:Weylpolar2} where we have defined Weyl-Papapetrou polar coordinates by $\rho = R\sin \Theta$ and $z= R \cos\Theta$. Then changing to Weyl polar coordinates in \eqref{eq:ALEbar} and \eqref{eq:ALEflat2} we deduce the result.

Now consider the ALF $N\neq 0$ case. Setting $\tau = N \psi$ in Proposition \ref{prop:ALFweyl} and defining $R$ with respect to Weyl-Papapetrou coordinates adapted to $\partial_\psi, \partial_\phi$ immediately implies the result. 
\end{proof}

\begin{remark}
Note that in the coordinate system of the previous proposition, the ALE case arises as the $N\to \infty$ limit of the ALF case.
\end{remark}

 It appears that the method used for for developing the asymptotic expansion for the ALF case, to obtain Proposition \eqref{prop:ALFweyl}, does not work in the ALE case. To see this, observe that ALE instantons can be written as in Proposition \ref{prop:ALFweyl} where  $\tau$ is $\psi$,  $V= R+O(1)$ and $\omega=O(1)$ so $\nabla Y=O(1)$, so the r.h.s. of \eqref{Weq} is $O(R^{-1})$ which is not fast enough for the results of Beig and Simon to apply. In any case, in what follows we will not need the detailed form of the asymptotic expansions.

\subsection{Asymptotic model}

For later purposes, it will be convenient to write an exact solution that captures the asymptotic metric in the ALE and generic $N
\neq 0$ ALF  case in a unified form. This is easily achieved by  the following Gibbons-Hawking metric 
\begin{equation}
\bar{\mathbf{g}}= H^{-1}( \td \psi+ \cos\theta \td \phi)^2 + H (\td r^2+ r^2 \td\theta^2+ r^2 \sin^2\theta \td \phi^2), \qquad H = c+ \frac{\varepsilon}{r}
\end{equation}
where $c\geq 0$ and $\varepsilon=\pm 1$.  For $c=1/N^2>0$ and $\varepsilon= \sgn(N)$ this is locally isometric to the Euclidean Taub-NUT metric in Example \ref{ex:TN} (possibly with $N<0$) as can be seen by rescaling $\tau= N \psi$ and $r$ (also it is the ALF asymptotic metric in Proposition \ref{prop:ALEFWeyl} with $W= H$ and $\Omega= \cos\theta$ and $R=r$). For $c=0$ and $\varepsilon=1$ this is a flat metric on $\mathbb{R}^4$ and corresponds to the ALE asymptotic metric \eqref{eq:ALEflat2}, as can be seen by changing coordinate $r= \tfrac{1}{4} \bar{r}^2$  (also it is the ALE asymptotic metric in Proposition \ref{prop:ALEFWeyl} with $W=1/R$ and $\Omega = \cos\theta$). Note that smoothness of the metric requires $H>0$, which is always the case for $\varepsilon=1$, whereas for $\varepsilon=-1$  requires $r>N^2$.  In both cases surfaces of constant $r$ are locally isometric to $S^3$.

We now extract the rod structure of this asymptotic model. In the $(\psi, \phi)$ coordinates the Gram matrix has determinant $\det \bar{g}= r^2 \sin^2\theta$. Hence, for $\varepsilon=1$ the axis divides into rods defined by $\theta=0, r>0$ where $\bar{v}_1:= \partial_\phi-\partial_\psi=0$ and $\theta=\pi, r>0$ where $\bar{v}_2:= \partial_\psi+\partial_\phi=0$ separated by a fixed point of the torus symmetry at $r=0$. If $\varepsilon=-1$ one has the same semi-infinite rods and rod vectors for $r>N^2$, although now the domain does not include the fixed point.  In any case, this will not matter as we will only use this metric for our asymptotic model at sufficiently large $r$. Therefore, for simplicity for the rest of this section we fix $\varepsilon=1$, although the same statements hold for $\varepsilon=-1$ provided $r$ is sufficiently large.

We can introduce angles adapted to the rod structure. These are given by the Hopf coordinates \eqref{eq:euler}, in terms of which 
 the rod vectors defined above are simply $\bar{v}_i= \partial_{\bar{\phi}^i}$ for $i=1,2$.  The metric in these coordinates  is
\begin{equation}\label{eq:barg}
\bar{\mathbf{g}}= H^{-1}(  (1+\cos\theta) \td \bar{\phi}^2 - (1-\cos\theta) \td \bar{\phi}^1)^2+ H r^2 \sin^2\theta (\td \bar{\phi}^1 + \td \bar{\phi}^2)^2 + H (\td r^2+ r^2 \td\theta^2)  \; .
\end{equation}
Note that for $c=0, \varepsilon=1$ this metric reduces to \eqref{eq:ALEflat} by the change of coordinates $(\bar r , \bar\theta)=( 2 \sqrt{r}, \tfrac{1}{2}\theta)$, as already observed.
 In any case, it is easy to see that \eqref{eq:barg} extends smoothly at $\theta=0, r>0$ and $\theta=\pi, r>0$ if $\bar{\phi}^i$ are both independently $2\pi$ periodic.  With this choice constant $r>0$ surfaces have $S^3$ topology and the metric extends smoothly to $r=0$ where there is a fixed point of the torus symmetry. It is useful to write \eqref{eq:barg} in Weyl-Papapetrou coordinates:  defining $\rho^2= \det \bar{g}$ with respect to the $\bar{\phi}^i$ coordinates we find
\[
\rho = 2r \sin \theta, \qquad z= 2r \cos\theta
\]
and the metric becomes
\[
\bar{\mathbf{g} }= \frac{( \mu^+\td \bar{\phi}^1- \mu^- \td \bar{\phi}^2)^2}{4 r^2 H}+ \frac{H \rho^2}{4} (\td \bar{\phi}^1 + \td \bar{\phi}^2)^2 + \frac{H}{4} (\td \rho^2+ \td z^2)
\]
where $\mu^\pm = \sqrt{\rho^2+z^2} \mp z$.  Now the rods are $\rho=0, z<0$ where $\mu^-=0$ so $\partial_{\bar{\phi}^2}=0$,  and $\rho=0, z>0$ where $\mu^+=0$ so $\partial_{\bar{\phi}^1}=0$.

In order to describe the most general topology at infinity we can now take a lens space quotient of the $S^3$ constant $r$ surfaces. As the metric is toric, this is simply effected by taking the quotient \eqref{eq:lens}. 
With these identifications constant $r$ surfaces are lens spaces $L(p,q)$ and $r=0$ now corresponds to an orbifold singularity. This singularity does not matter though,  as we will only use this as an asymptotic model geometry for sufficiently large $r$.  We can introduce coordinates that are independently $2\pi$-periodic by 
\begin{equation}
    \phi^1= p \bar{\phi}^1, \qquad  \phi^2 = \bar{\phi}^2- q \bar{\phi}^1 
\end{equation}
so the identifications on $\bar{\phi}^i$ are equivalent to 
\[
(\phi^1, \phi^2)\sim (\phi^1+2\pi , \phi^2), \qquad (\phi^1, \phi^2)\sim (\phi^1, \phi^2+2\pi)\;. 
\]
We can write the Gram matrix in the $\phi^i$ coordinates as
\begin{equation}
    g= L \bar{g} L^T, \qquad \bar{g} = \begin{pmatrix}
        \frac{(\mu^+)^2}{4r^2 H}+ \frac{H \rho^2}{4} & \frac{H\rho^2}{4} \left( 1 - \frac{1}{ r^2 H^2} \right) \\ \frac{H \rho^2}{4} \left( 1 - \frac{1}{r^2 H^2 } \right)&  \frac{(\mu^-)^2}{4r^2 H}+ \frac{H \rho^2}{4}
    \end{pmatrix} \; , \qquad L= \begin{pmatrix}
        \frac{1}{p} & \frac{q}{p} \\ 0 & 1 
    \end{pmatrix}
    \label{eq:modelALF}
\end{equation}
where we have used $\partial_{\phi^i}= L_{ij} \partial_{\bar{\phi}^j}$. Note that $\det g= p^{-2} \rho^2$  and we can verify
\[
L^T \begin{pmatrix}
    0 \\ 1
\end{pmatrix} = \begin{pmatrix} 0 \\1 \end{pmatrix}, \qquad L^T \begin{pmatrix}
    p \\ -q
\end{pmatrix} = \begin{pmatrix}
    1 \\0
\end{pmatrix}
\]
which implies 
\begin{equation}
    g|_{\rho=0, z<0}\begin{pmatrix}
    0 \\ 1
\end{pmatrix} =0, \qquad g|_{\rho=0, z>0}\begin{pmatrix}
    p \\ -q
\end{pmatrix}=0
\end{equation}
so the rod vectors on the two semi-infinite rods are indeed those for a general lens space $L(p,q)$.

\section{Harmonic map formulation}
\label{sec3}

\subsection{Uniqueness and existence}
We may reformulate toric gravitational instantons as harmonic maps, following~\cite{Kunduri:2021xiv}.  Define $\Phi = \rho^{-1} g$ so that $\det \Phi=1$.   Then, by harmonicity of $\rho$,
\begin{equation}
     \td (\rho \hat\star \Phi^{-1} \td \Phi)=0  \; .
\end{equation}
In terms of Weyl-Papapetrou coordinates this can be written as
\begin{equation}
    \nabla \cdot (\Phi^{-1} \nabla \Phi)=0
\end{equation}
where $\nabla$ is the gradient operator on $\mathbb{R}^3$ with Euclidean metric $\delta= \td \rho^2+ \rho^2 \td \varphi^2+ \td z^2$ and $\Phi$ is an axisymmetric function on $\mathbb{R}^3 \backslash \Gamma$ where $\Gamma$ is the $z$-axis (note $\varphi$ here is an auxiliary coordinate and axisymmetry of $\Phi$ means $\partial_\varphi \Phi=0$). Thus $\Phi : \mathbb{R}^3\backslash \Gamma \to  N$ is a harmonic map, where $N= SL(2, \mathbb{R})/SO(2)$ is a symmetric space.  The tension of such a map is defined by 
\[
\tau(\Phi)= \Phi \nabla \cdot (\Phi^{-1} \nabla \Phi)
\]
and hence it is harmonic iff the tension vanishes. Its norm squared with respect to the target space metric is
\begin{equation}
    | \tau(\Phi)|^2 = \text{Tr}( \Phi^{-1} \tau \Phi^{-1} \tau) = \text{Tr}[(\nabla \cdot (\Phi^{-1} \nabla \Phi))^2 ]   \label{eq:tensionsq}
\end{equation}
Therefore, we wish to address uniqueness and existence of such harmonic maps with a prescribed (singular) boundary conditions on the axis $\Gamma$.  Fortunately, there is a well developed theory for harmonic maps with target spaces that include $N$, which we may exploit.

First, we establish the uniqueness part of Theorem \ref{th:exist}. This is much simpler and is essentially the same argument used to prove the black hole uniqueness theorem, see e.g.~\cite{Hollands:2008fm}. 

\begin{theorem}
    There is at most one toric ALE and toric ALF instanton for any given admissible rod structure and for given $N\neq 0$ in the ALF case.
\end{theorem}

\begin{proof}
    Let $\Phi, \tilde\Phi: \mathbb{R}^3\backslash \Gamma \to  N$ be two harmonic maps and define the Mazur distance between these $\Psi:= \text{Tr} (\tilde \Phi \Phi^{-1}- I)$, which is nonnegative and vanishes iff $\Phi=\tilde\Phi$.  The Mazur identity then states that on $\mathbb{R}^3\backslash \Gamma $ (see e.g. ~\cite{Kunduri:2021xiv}),
    \begin{equation}
        \Delta_3 \Psi \geq 0  \; .
    \end{equation}
Now suppose the harmonic maps have the same rod structure. Then, for any rod $I_i$ we may choose a basis adapted to the corresponding rod vector and parameterise the Gram matrix as \eqref{eq:gparam}. A computation in this basis shows that 
\begin{equation}
\Psi= \frac{ (h-\tilde h)^2  +\rho^2 (w - 
\tilde w)^2}{ h \tilde h}
\end{equation}
where  $h, \tilde h >0$ and $w, \tilde w$ are all functions that are smooth in $\rho^2, z$ (the tildes refer to the parameterisation of $\tilde\Phi$).  Therefore, $\Psi$ extends to smooth function on the $z$-axis and hence defines a subharmonic function on all of $\mathbb{R}^3$.

On the other hand, since the two solutions also share the same asymptotics, we must have $\Psi \to 0$ at infinity.  To see this explicitly, consider the Gram matrix relative to the basis $(\partial_\psi, \partial_\phi)$ in Proposition \ref{prop:ALEFWeyl}. Then, 
\begin{equation}
    \Psi= \frac{ (W-\tilde W)^2 + \rho^{-2} (\Omega - \tilde \Omega)^2}{W \tilde W}  \; ,
\end{equation}
where the tilded quantities refer to the data corresponding to $\tilde \Phi$.
Hence, by  Proposition \ref{prop:ALEFWeyl}, we deduce that 
\begin{equation}
    \Psi= \begin{cases}
        O(R^{-\tau})  \\   \frac{ (N^2-\tilde N^2)^2}{N^2 \tilde N^2}  + O(R^{-1}) 
    \end{cases}
    \quad \begin{array}{c}
        \text{ALE} \\ \text{ALF}
    \end{array}
\end{equation}
as $R\to \infty$.
Hence for the ALE case $\Psi\to 0$ automatically. On the other hand, for the ALF case we must fix the asymptotic invariant $| N|=|\tilde N |$, in which case one has $\Psi= O(R^{-2})$ and hence also $\Psi \to 0$.  Therefore, by the maximum principle $\Psi$ vanishes identically and hence $\Phi= \tilde\Phi$.    
\end{proof}

To establish the existence part of Theorem \ref{th:exist} we can exploit  deeper results on harmonic maps  previously developed by Weinstein for stationary and axysmmetric spacetimes.

\begin{definition}
A map $\Phi_0: \mathbb{R}^3\backslash \Gamma \to N$ with $|\tau(\Phi_0)|$  bounded that decays at infinity 
sufficiently fast is called a {\it model map}. Two maps $\Phi_1, \Phi_2: \mathbb{R}^3\backslash \Gamma \to N$ are said to be {\it asymptotic} if $\text{dist}_{N}(\Phi_1, \Phi_2)$ is  bounded and decays to zero at infinity. 
\end{definition}

\begin{theorem}[Weinstein~\cite{Weinstein:2019zrh}] 
\label{th:weinstein} Given a  model map $\Phi_0: \mathbb{R}^3\backslash \Gamma \to N$ there exists unique harmonic map $\Phi : \mathbb{R}^3\backslash \Gamma \to N$ which is asymptotic to $\Phi_0$.
\end{theorem}

In what follows,  we will write down a model map that exhibits any admissible rod structure for either ALE or ALF asymptotics.  In order to complete the proof of the main theorem we also need to show that the corresponding harmonic map shares the same rod structure and asymptotics as the model map. This is achieved by the following result.
\begin{proposition}
    If a harmonic map $\Phi$ is asymptotic to an ALE or ALF model map $\Phi_0$, then it has the same rod structure  as $\Phi_0$ and is also ALE or ALF.
\end{proposition}

\begin{proof}
    The proof that $\Phi$ has the same rod structure as $\Phi_0$ is the same as in the AF case~\cite{Kunduri:2021xiv}. It remains to show $\Phi$ has the claimed asymptotics. We note that $\text{dist}_N(\Phi, \Phi_0)$ is equivalent to the Mazur distance $\text{Tr}(\Phi \Phi_0^{-1}-I)$. Therefore, since $\Phi \Phi_0^{-1}= g g_0^{-1}$, we can easily compute this in the parameterisation of the Gram matrix  \eqref{eq:Weylpolar2} which gives, 
    \begin{equation}
    \text{Tr}( \Phi \Phi_0^{-1}-I)= \frac{(W-W_0)^2+ \rho^{-2}( \Omega-\Omega_0)^2}{W W_0} \; .
    \end{equation}
    Therefore, if $\text{Tr}(\Phi \Phi_0^{-1}-I) \to 0$ as $R= \sqrt{\rho^2+z^2} \to \infty$ for all $0\leq \Theta \leq \pi$, then $W \to W_0$ and $\Omega\to \Omega_0$ for both asymptotic types. Hence, if $\Phi_0$ is ALE or ALF then so is $\Phi$.
\end{proof}

\begin{remark}
    Strictly speaking to show $\Phi$ is ALE and ALF according to our definitions we also need to control the rate of decay of the subleading terms.  This requires a detailed asymptotic analysis of harmonic maps with prescribed singularities, which in fact has only been recently completed for the case of asymptotically flat stationary and axisymmetric spacetimes~\cite{Han:2022mmc}. Presumably an analogous analysis can be performed for toric gravitational instantons, but we will not pursue this here.
\end{remark}

Therefore, the existence part of Theorem \ref{th:exist} reduces to exhibiting a model map with the required asymptotics, which we turn to next.  Before doing so, we remark that a corollary of the uniqueness and existence theorem is that a toric instanton with $n-1$ finite rods and given rod vectors is uniquely specified by $n-1$ continuous parameters, namely the finite rod lengths $\ell_i= z_i-z_{i-1}$ for $i=2, \dots, n$ (only differences of the $z$-coordinate have invariant meaning), and an additional asymptotic parameter $N\neq 0$ for the ALF case. On the other hand, one expects that the condition for the absence of a possible conical singularity over each finite rod removes $n-1$ parameters. Therefore,  generically, one expects that the moduli space of smooth toric instantons  with given rod vectors, is at most $0$-dimensional in the ALE case and $1$-dimensional for the ALF $N\neq 0$ case.  In contrast, in the AF case the same argument shows the moduli space is $2$-dimensional, since then one has two continuous asymptotic parameters $\beta, \Omega$~\cite{Kunduri:2021xiv}. The parameter $\beta$ encodes the norm of the Killing vector that is bounded at infinity, which in the ALF case is given by $N$ (there is no analogue of $\Omega$ for the $N\neq 0$ ALF case since the topology $L(p,q)$ fixes the identifications of the torus angles).

\begin{remark}
 In light of the above parameter counting, it is interesting to consider various examples. The Taub-Bolt instanton is a 1-parameter family of smooth solutions with a single finite rod, which arises from the conically singular 2-parameter family of Schwarzschild-Taub-NUT solutions,  in line with the above expectation.  However, the self-dual multi-Eguchi-Hanson and multi-Taub-NUT instantons are smooth ALE and ALF instantons with $n-1$ finite rods that depend on $n-1$ and $n$ parameters respectively (see Section \ref{sec:SD}). Therefore, it appears that in the self-dual case the conical singularities are automatically absent and these instantons are the unique solutions of Theorem \ref{th:exist} possessing their rod structure that are ALE and ALF.
\end{remark}

\subsection{Model map}

We now write down an explicit model map that exhibits any admissible rod structure for either ALE or ALF asymptotics. The construction is essentially the same as the AF case in the interior region~\cite{Kunduri:2021xiv}, but is more complicated in the asymptotic end and the adjacient transition regions.

\begin{prop} \label{prop:modelmap}
There exists an ALE and ALF model map which exhibits any admissible rod structure.
\end{prop}

\begin{proof}
Consider the general case with $n$ corners, so that there are $n+1$ rods $I_i$.  We will  divide the $(\rho, z)$ half-plane into several regions.
On the interior of each rod we define a transition region $T_i \subset  I_i$ and thicken them into regions $\mathcal{T}_i= T_i \times [0,\rho_0)$ for some $\rho_0>0$.  Next, define  $\mathcal{R}$ to be the region  $|z| < R$ and $0\leq \rho < \rho_0$.  Then $\mathcal{R} \setminus \left(\cup_i\mathcal{T}_i \right)$ leaves $n+2$ regions $\mathcal{S}_i$ separated by the $n+1$ transition regions $\mathcal{T}_i$. We denote the intersection of these regions with the axis by $S_i = \mathcal{S}_i \cap \{ \rho=0 \}$. Note that for $i=2, \dots, {n+1}$ the regions $\mathcal{S}_i$ include  part of  the rods $I_{i-1}$ and $I_i$ and the corner $z_{i-1}$ that separates them, whereas $\mathcal{S}_1$ and $\mathcal{S}_{n+2}$ sit over $I_1$ and $I_{n+1}$ respectively.  Note also that the transition regions $\mathcal{T}_i$ do not contain any of the corner points $z_i$. We also define a region near infinity $C_R$ by  $\rho^2+z^2 > R^2$. See Figure \ref{fig:model} for a depiction of these regions.

\begin{figure}[h!]
\centering
{
\begin{tikzpicture}[scale=1.5, every node/.style={scale=0.6}]
\fill[ fill=black, opacity=0,very thick](-4,4)--(-4,0)--(4,0)--(4,4);
\draw[black,thick](-4.2,0)--(-1.8,0);
\draw[black,thick](1.8,0)--(4.2,0)node[black,font=\large,right=.2cm]{\color{black}{$z$}};
\draw[black,thick,dashed](-1.8,0)--(1.8,0);
\draw[black,thick](0,0)--(0,4.2)node[black,font=\large,above=.2cm]{\color{black}{$\rho$}};
\draw[black,fill=black] (-2.4,0) circle [radius=.07] node[black,below=.4cm]{$z_1$};
\draw[black] (-1.5,1.2) node[black,above=.4cm]{};
\draw[black] (-1.5,3.6) node[black,above=.4cm]{$C_R$};
\draw[black,fill=black] (-3.8,0) circle [radius=.03] node[black,below=.4cm]{$-R$};
\draw[black,fill=black] (3.8,0) circle [radius=.03] node[black,below=.4cm]{$R$};
\draw[gray,dashed](-3.8,1.2)--(3.8,1.2);
\draw[gray](-3.8,0)--(-3.8,1.2);
\draw[gray](3.8,0)--(3.8,1.2);
\draw[black] (-3.4,0) node[black,above=.8cm]{$\mathcal{S}_{1}$};
\draw[gray](-3.2,0)--(-3.2,1.2);
\draw[black] (-3.0,0) node[black,above=.8cm]{$\mathcal{T}_{1}$};
\draw[gray](-2.6,0)--(-2.6,1.2);
\draw[black] (-2.4,0) node[black,above=.8cm]{$\mathcal{S}_{2}$};
\draw[gray](-2.2,0)--(-2.2,1.2);
\draw[black] (-2.1,0) node[black,above=.8cm]{$\mathcal{T}_{2}$};
\draw[gray](-1.95,0)--(-1.95,1.2);
\draw[black] (-1.8,0) node[black,above=.8cm]{$\mathcal{S}_{3}$};
\draw[black,fill=black] (-1.8,0) circle [radius=.07] node[black,below=.4cm]{$z_2$};

\draw[gray](-1.6,0)--(-1.6,1.2);
\draw[gray](1.6,0)--(1.6,1.2);
\draw[black] (1.75,0) node[black,above=.8cm]{$\mathcal{S}_{n}$};
\draw[black,fill=black] (1.8,0) circle [radius=.07] node[black,below=.4cm]{$z_{n-1}$};
\draw[black] (2.0,0) node[black,above=.8cm]{$\mathcal{T}_{n}$};
\draw[gray](1.9,0)--(1.9,1.2);
\draw[black] (2.4,0) node[black,above=.8cm]{$\mathcal{S}_{n+1}$};
\draw[black] (1.35,1.2) node[black,above=.4cm]{};
\draw[gray](2.6,0)--(2.6,1.2);
\draw[black] (2.9,0) node[black,above=.8cm]{$\mathcal{T}_{n+1}$};
\draw[gray](3.2,0)--(3.2,1.2);
\draw[black] (3.4,0) node[black,above=.8cm]{$\mathcal{S}_{n+2}$};
\draw[black](0,0) node[black,below=.4cm]{\color{black}{}};
\draw[gray](2.2,0)--(2.2,1.2);
\draw[black,fill=black] (2.4,0) circle [radius=.07] node[black,below=.4cm]{$z_{n}$};
\draw[black] (3.8,0) arc (0:180:3.8cm);
\end{tikzpicture}}
\caption{Regions for the model map.\label{fig:model} }
\end{figure}
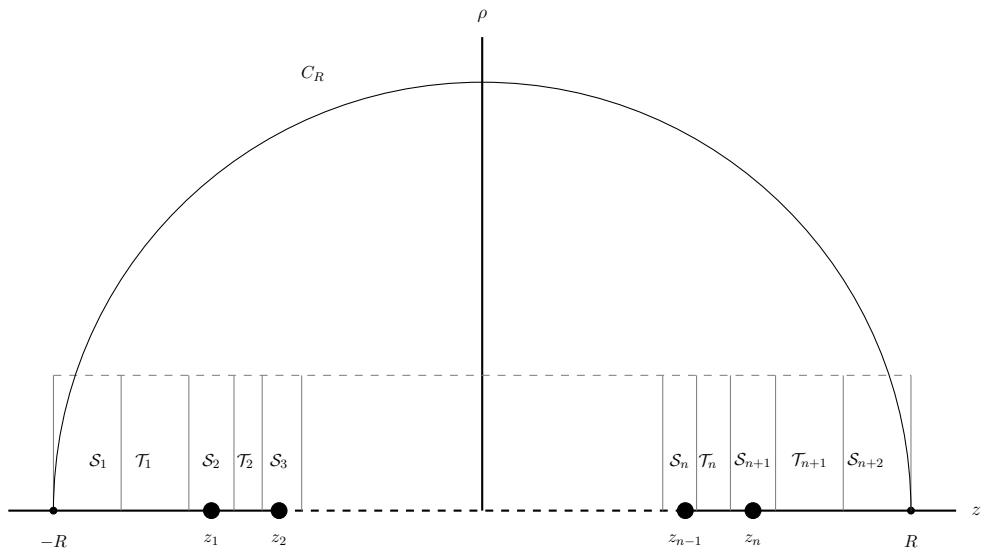

 We will define the model map $\Phi_0$ by specifying it on each of these regions. On the  remaining compact region $\rho^2+z^2\leq R^2$ and $\rho\geq \rho_0$ we take $\Phi_0$ to be any smooth extension so that this defines the model map everywhere.   It will be useful to define the following functions  
\begin{equation}
\mu_i^\pm := \sqrt{\rho^2+ (z-z_i)^2} \mp (z-z_i)
\end{equation}
which have the property that $\mu_i^+=0$ for $\rho=0, z>z_i$ and $\mu_i^-=0$ for $\rho=0, z<z_i$ and satisfy $\Delta_3 \log \mu_i^\pm=0$.

We are now ready to specify our model map. We will work in the basis of torus Killing fields defined by our asymptotic model \eqref{eq:modelALF} which in particular requires $\det g= p^{-2} \rho^2$.  First we describe the interior regions which are common to both the ALE and ALF case. We then describe the asymptotic ALE or ALF region in unified manner (we give an alternate simpler description for the ALE case in the Appendix). Finally, we verify the tension of our map is indeed bounded.  \\

\paragraph{\it Interior region}

On the regions $\mathcal{S}_i$, $i=2, \dots, n+1$, take
\begin{equation}\label{eq:modmapinterior}
g =  M_{i-1} D_{i-1} M_{i-1}^T , \qquad D_{i-1} := \begin{pmatrix} \mu_{i-1}^+ & 0 \\  0 & \mu_{i-1}^- \end{pmatrix} , \qquad  M_{i-1} := \begin{pmatrix}  p^{-1} v^2_{i-1} & \epsilon_{i-1} v^2_{i} \\ -p^{-1} v^1_{i-1}  & -\epsilon_{i-1} v^1_{i} \end{pmatrix}
\end{equation}  
Observe that 
\[
M_{i-1}^{T} \begin{pmatrix} v^1_i \\ v^2_i \end{pmatrix}= - \epsilon_{i-1} p^{-1} \begin{pmatrix} 1 \\ 0 \end{pmatrix}  \; ,\qquad M_{i-1}^{T} \begin{pmatrix} v^1_{i-1} \\ v^2_{i-1} \end{pmatrix} =\begin{pmatrix} 0 \\ 1 \end{pmatrix}  \; ,
\]
which implies that 
if $z \in {S}_i$, 
\begin{equation}
\text{ker}(g)_{\rho =0} = \begin{cases} \text{span}(\underline{v}_{i-1}) & z < z_{i-1} \\ \text{span}(\underline{v}_i) & z > z_{i-1} \end{cases}
\end{equation} so these indeed exhibit the required rod structure in these regions. Also  $\det M_i=p^{-1}$ 
  in view of the admissibility condition and hence $\det g =p^{-2}\rho^2$.   It is useful to note that on these regions we can also  write 
\begin{equation}
g = \tilde{M}_{i-1}\tilde{D}_{i-1} (\tilde{M}_{i-1})^T ,\qquad \tilde{D}_{i-1}:=  \begin{pmatrix} \mu_{i-1}^- & 0 \\  0 & \mu_{i-1}^+ \end{pmatrix},  \qquad \tilde{M}_{i-1} = \begin{pmatrix} v^2_i& -\epsilon_{i-1}p^{-1} v^2_{i-1} \\ -v^1_{i} &   \epsilon_{i-1} p^{-1} v^1_{i-1} \end{pmatrix}
\end{equation} 
where the two forms are related by $\tilde{M}_{i-1} = M_{i-1} C_{i-1}$  and $C_{i-1}= \epsilon_{i-1}\sigma$ where $\sigma$ is antisymmetric with $\sigma_{12}=-1$.

On the transition regions $\mathcal{T}_i$, $i=2, \dots, n$, which separate $\mathcal{S}_i$ and $\mathcal{S}_{i+1}$, take
\begin{equation}\label{g_T}
g = N_i(z) \tilde{D}_{i-1}^{\lambda(z)} D_i^{1-\lambda(z)} N_i(z)^T, 
\end{equation} where
\begin{equation}
N_i (z) = \begin{pmatrix} \gamma(z) v^2_i & \alpha^2_i(z) \\  -\gamma(z) v^1_i  & -\alpha^1_i(z)\end{pmatrix},\qquad \underline{\alpha}_i(z) :=  p^{-1} \gamma(z)^{-1} \left( -\epsilon_{i-1} \lambda(z) \underline{v}_{i-1} + \epsilon_i \underline{v}_{i+1} (1-\lambda(z)) \right) \; ,
\end{equation} and 
we have defined smooth  functions $\lambda(z), \gamma(z)$ such that  $0\leq \lambda(z)\leq 1$ and $
\gamma(z)>0$ and 
\[
\lambda(z) = \begin{cases} 1 \\ 0 
\end{cases}  \qquad \gamma(z) =\begin{cases}
    1  \\ p^{-1}  
 \end{cases}
 \qquad \qquad \text{if} 
\qquad 
 \begin{array}{c} z\in S_i \\  z\in S_{i+1} \end{array} \;.
\]
This ensures that $N_i$ interpolates from $\tilde{M}_{i-1}$ for $z\in S_i$ to $M_i$ for $z\in S_{i+1}$ and hence $g$ interpolates between its values on $\mathcal{S}_i$ and $\mathcal{S}_{i+1}$.  For $z \in {T}_i$, we have $\det N_i=p^{-1}$ and
\begin{equation}
g|_{\rho =0} = N_i(z)\begin{pmatrix} 4 |z-z_{i-1}|^{\lambda(z) }| z-z_i|^{1-\lambda(z)} & 0 \\ 0 & 0 \end{pmatrix} N_i(z)^{T}
\end{equation}
and since 
\begin{equation}\label{N_A}
N_i(z)^T \begin{pmatrix} v_i^1 \\ v_i^2 \end{pmatrix} =\det(\underline{v}_i, \underline{\alpha}_i) \begin{pmatrix} 0 \\ 1 \end{pmatrix} = p^{-1} \gamma(z)^{-1}\begin{pmatrix} 0 \\ 1 \end{pmatrix} ,
\end{equation}
we deduce 
\begin{equation}
\text{ker} (g)_{\rho =0} = \text{span}(\underline{v}_i) ,
\end{equation}
so we have the correct rod structure on the transition region $T_i$. \\

\paragraph{\it Asymptotic ALE or ALF end}
It remains to specify $g$ on the regions $\mathcal{S}_1, \mathcal{T}_1, \mathcal{T}_{n+1}\mathcal{S}_{n+2}$ and on $C_R$. These now depend on the ALE or ALF condition. We will treat these simultaneously using the asymptotic model developed earlier (in the Appendix we give an alternative simpler model map for the ALE case).
Thus on $C_R, \mathcal{S}_1, \mathcal{S}_{n+2}$ we take the model map to be given by the asymptotic model \eqref{eq:modelALF} so the rod vectors on the semi-infinite rods are
\begin{equation}
\underline{v}_1= \begin{pmatrix}
    0 \\ 1
\end{pmatrix}, \qquad \underline{v}_{n+1}= \begin{pmatrix}
    p \\ -q
\end{pmatrix}
\end{equation}
In order to ensure that these are the rod vectors on $I_1$ and $I_{n+1}$ we use the translation freedom in $z$ to set $z_1= 0$, so $I_1$ corresponds to $z<0$ and $I_{n+1}$ to $z>z_n>0$.  Also note that we must take $C_R, \mathcal{S}_1, \mathcal{S}_{n+1}$ to be at sufficiently large $r$ so as to ensure that $H>0$ in the asymptotic model \eqref{eq:modelALF} (this is only relevant in the ALF case with $N<0$), which is always possible to achieve by taking the constant $R$ and the transition regions $\mathcal{T}_1, \mathcal{T}_{n+1}$ large enough compared to $1/N^2$.
It remains to consider the transition regions. 

First, on $\mathcal{T}_1$, which separates $\mathcal{S}_1$ and $\mathcal{S}_2$, we write
 \begin{equation}
 g =N_1(z) H_1 N_1(z)^T, \qquad H_1 = \begin{pmatrix}
    A & \rho^2 B \\ \rho^2 B & \frac{\rho^2+\rho^4 B^2}{A} 
 \end{pmatrix}, \qquad  N_1(z)= \begin{pmatrix}
         p^{-1} & \beta(z) \\ 0 & 1
    \end{pmatrix}  \; ,
 \end{equation}
 where $\beta(z), A(\rho, z), B(\rho, z)$ are smooth functions such that $A>0$ and  
 \[
\beta(z) =\begin{cases}
    \frac{q}{p} \\ \epsilon_1 v_2^2
\end{cases} \qquad A= \begin{cases}
    \bar{g}_{11} \\ \mu_1^+ 
\end{cases}
\qquad 
B= \begin{cases}
    \frac{\bar{g}_{12}}{\rho^2} \\ 0 
\end{cases}
\qquad \qquad \text{if}
\qquad 
 \begin{array}{c} z\in S_1 \\  z\in S_{2} \end{array} \;,
 \]
 where $\bar{g}_{ij}$ are the components of the Gram matrix of the asymptotic model \eqref{eq:modelALF} and 
 \begin{equation}
 \frac{\bar{g}_{12}}{\rho^2} = \frac{H}{4} \left( 1 - \frac{1}{ r^2 H^2} \right) \label{eq:g12}
 \end{equation} is a smooth function on $\mathcal{S}_1$.  Note that we can choose $A$ to be strictly positive since it interpolates between two strictly positive functions (recall $\mu_1^+>0$ for $z<z_1$). This ensures that $N_1$ interpolates between $L$ and $M_1$, which is 
 \[
 M_1 = \begin{pmatrix}
     p^{-1} & \epsilon_1 v_2^2 \\ 0 & 1 
 \end{pmatrix}  \; ,
 \]
 and that $H_1$ goes between  $\bar{g}$ and $D_1$. Thus this choice smoothly interpolates between the values of $g$ on $\mathcal{S}_1$ and $\mathcal{S}_2$ defined above. Furthermore, $\det g = p^{-2} \rho^2$ and $N_1^T \underline{v}_1= \underline{v}_1$ so $v_1 \in \text{ker} g|_{\rho=0}$ so $v_1$ is the rod vector on the whole transition region ${T}_1$.

 Second, on the transition region $\mathcal{T}_{n+1}$, separating $\mathcal{S}_{n+1}$ and $\mathcal{S}_{n+2}$, take
 \[
 g = N_{n+1}(z) H_{n+1} N_{n+1}(z)^T , \qquad H_{n+1}= \begin{pmatrix}
  \frac{\rho^2+\rho^4 B^2}{A}    & \rho^2 B \\ \rho^2 B & A
 \end{pmatrix}, \qquad  N_{n+1}(z)= \begin{pmatrix}
         \frac{1}{p a}+ \frac{q c}{p}  & \frac{qa}{p}  \\ c & a 
     \end{pmatrix}
 \]
 for smooth functions $A(\rho, z), B(\rho,z), a(z), c(z)$ such that $A>0, a>0$, where we have chosen a parameterisation to ensure $\det H_{n+1}=\rho^2$ and $\det N_{n+1}=p^{-1}$.  We require that 
\[
a(z) =\begin{cases}
    -\epsilon_n p \\ 1
\end{cases}
\qquad
c(z) = \begin{cases}
-p^{-1} v_n^1 \\ 0
\end{cases}  \; ,\qquad \qquad \text{if}
\qquad 
 \begin{array}{c} z\in S_{n+1} \\  z\in S_{n+2} \end{array} \;,
 \]
 to ensure that $N_{n+1}$ interpolates from 
 \[
 M_{n+1}= \begin{pmatrix}
     p^{-1} v_n^2 & - \epsilon_n q \\ - p^{-1} v_n^1 & -\epsilon_n p 
 \end{pmatrix}
 \]
 on $\mathcal{S}_{n+1}$ to $L$ on $\mathcal{S}_{n+2}$, and the requirement $a>0$ implies we must fix $\epsilon_n=-1$ (recall $p>0$). We also require 
 \[
 A= \begin{cases}
     \mu_n^- \\ \bar{g}_{22}
 \end{cases} \qquad B= \begin{cases}
     0 \\ \frac{\bar{g}_{12}}{\rho^2}
 \end{cases}
 \qquad \qquad \text{if}
\qquad 
 \begin{array}{c} z\in S_{n+1} \\  z\in S_{n+2} \end{array} \;,
 \] 
 to ensure that $H_{n+1}$ goes from $D_n$ to $\bar{g}$ the asymptotic model \eqref{eq:modelALF}, where \eqref{eq:g12} is also smooth on $\mathcal{S}_{n+1}$.  Note we can choose $A$ to be strictly positive since it interpolates between two strictly positive functions (recall $\mu_n^->0$ for $z>z_n$).  Thus this choice smoothly interpolates between $g$ on $\mathcal{S}_{n+1}$ and $\mathcal{S}_{n+2}$. Again note we have $\det g = p^{-2} \rho^2$ and 
 \[
N_{n+1}^T \begin{pmatrix} p \\ - q \end{pmatrix} = a^{-1} \begin{pmatrix}
    1 \\  0
\end{pmatrix}  \; ,
 \]
 so $v_{n+1} \in \text{ker}\, g|_{\rho=0}$ which implies $v_{n+1}$ is the rod vector on the whole of the transition region $\mathcal{T}_{n+1}$.\\ 

\paragraph{\it Boundedness of tension} The requirement that $| \tau |$ is bounded is satisfied if the components of the matrix $E:= \nabla \cdot (g^{-1}\nabla g)$ are bounded, as can been seen from equation \eqref{eq:tensionsq} and the identity $\nabla\cdot (\Phi^{-1}\nabla \Phi)=\nabla \cdot ( g^{-1}\nabla g)$.  Recall, that in the compact region $\rho^2+z^2\leq R^2$ and $\rho\geq \rho_0$ the model map is any smooth extension of the map defined on the other regions, hence by continuity $E$ must be bounded in this region. In the regions $C_R, \mathcal{S}_1, \mathcal{S}_{n+2}$, the model map is taken to be \eqref{eq:modelALF} which is a solution and hence has vanishing tension.  Similarly, the  map on the regions $\mathcal{S}_{i=2, \dots, n+1}$ is a solution so the tension also vanishes here. Thus one needs to check boundedness of the tension on the transition regions. In fact, this was checked for $\mathcal{T}_{i=2, \dots, n}$ in the AF case~\cite{Kunduri:2021xiv}. Therefore, it remains to check the regions $\mathcal{T}_1$ and $\mathcal{T}_{n+1}$. By explicit calculation one can check that the components of $E$ on both of these transition regions are indeed bounded as $\rho \to 0$ provided the functions $A$ and $B$ are smooth functions of $\rho^2$ (since this guarantees $\partial_\rho A=0= \partial_\rho B$ at $\rho=0$). In particular,  on $\mathcal{T}_1$ one finds that as $\rho\to 0$
\begin{gather}
E^1_{~1}=- E^2_{~2}= \frac{\partial_\rho A}{A \rho}+ O(1),\qquad E^1_{~2}= \frac{3p \partial_\rho (A B) \rho}{A^2} + O(\rho^2) \\\ E^2_{~1}= \frac{-2p \beta \partial_\rho A + (\partial_\rho A) A B +3 A^2 \partial_\rho B}{Ap \rho} + O(1) \; ,
\end{gather}
so if we choose $A$ and $B$ to be smooth functions of $\rho^2$ then $\partial_\rho A= O(\rho)$ and $\partial_\rho B = O(\rho)$ and the tension is indeed bounded. On $\mathcal{T}_{n+1}$ similar expressions hold and again choosing $A$ and $B$ to be smooth functions of $\rho^2$ implies $E^i_{~j}= O(1)$, so the tension is also bounded in this region.
\end{proof}

\section{Self-dual instantons}
\label{sec:SD}

In this section we consider gravitational instantons that are self-dual, that is, have a Riemann tensor that is self-dual.  In the ALF-$A_k$ case Minerbe has shown that these are exhausted by the the multi-Taub-NUT instantons~\cite{Minerbe:2009gj}. On the other hand, Kronheimer's construction shows that in the ALE case with $\Gamma$ cyclic these correspond to the multi-Eguchi-Hanson instantons. These results do not assume the existence of isometries and show that the metric must be of Gibbons-Hawking type. On other hand, it is well known that the local form of a toric self-dual instanton is given by the Gibbons-Hawking metrics.  We will now perform a direct global analysis of this family of metrics to establish an elementary proof of the aforementioned results in the toric setting.

\begin{proposition}
    Any self-dual toric ALE or ALF instanton $(M, g)$ is given by a multi-Eguchi-Hanson or multi-Taub-NUT metric. 
\end{proposition}
\begin{proof}
    In Weyl-Papapetrou coordinates $(\rho,  z, \tilde \phi^i)$ a Ricci-flat toric K\"ahler metric can be written as an axisymmetric Gibbons-Hawking metric 
\begin{equation}\label{eq:GH}
\mathbf{g} = H^{-1}( \td \tilde{\phi}^2+ \chi \td \tilde \phi^1)^2+ H ( \td \rho^2+ \td z^2+ \rho^2 (\td \tilde{\phi}^1)^2)
\end{equation}
where $\td \chi = \hat\star \rho \td H$, so $H$ is an axisymmetric harmonic function in $\mathbb{R}^3$~\cite{Kunduri:2021xiv}. In order to arrive at this form of the metric a $GL(2, \mathbb{R})$ transformation on the torus coordinates must be performed and hence $\tilde \phi^i$ are not necessarily $2\pi$-periodic.  

Observe that the norm of the Killing field $\partial_{\tilde\phi^2}$ is $H^{-1}$. This implies that $H^{-1}$ must be a smooth nonnegative function on $M$, so $H$ is smooth and positive wherever $|\partial_{\tilde\phi^2}|>0$ and is singular at any fixed points of $\partial_{\tilde\phi^2}$.  Near each rod $I_i$ smoothness requires that as $
H^{-1}= f(z) + O(\rho^2)$ as $\rho \to 0$, 
where $f(z)$ is a smooth non-negative function on the interior of each rod.  On the other hand, comparing to  the Gibbons-Hawking metric \eqref{eq:GH} shows that the conformal factor $e^{2\nu}= c H$  where $c>0$ is a constant depending on the $GL(2, \mathbb{R})$ transformation required to obtain the torus coordinates $\tilde{\phi}^i$. Hence, from \eqref{eq:nunearI} we deduce that $H$ is smooth and positive on the interior of any rod, so $f(z)$ is a strictly positive function on the interior of each rod.  However, since $\partial_{\tilde\phi^2}$ must vanish at the fixed points of the torus symmetry, we deduce that the zeros of $\partial_{\tilde\phi^2}$ precisely correspond to the corners of the orbit space $z=z_i$ and $f(z)$ vanishes at these points.  This implies that $H$ must have isolated singularities at $z=z_i$ and is smooth everywhere else. 

From the definition of $\chi$  it follows that near each rod $\partial_z \chi = O(\rho^2)$ and hence we can define constants $\chi_i:= \chi|_{I_i}$. Therefore, the $2\pi$-periodic rod vectors are $v_i = \partial_{\tilde\phi^1}- \chi_i \partial_{\tilde\phi^2}$ where we have chosen the normalisation to ensure the conical singularity in \eqref{eq:GH}  as $\rho \to 0$ over $I_i$ is absent.  The metric induced on the corresponding 2-surface is 
\[
\frac{\td z^2}{f(z)}+ f(z)\td \xi_i^2
\]
where $\xi_i := \tilde{\phi^2}+ \chi_i \tilde{\phi}^1$ and the condition for this to extend to a smooth metric on $S^2$ is that $\xi_i$ is identified with period $\ell_i=4\pi/ f'(z_{i-1})$ and $f'(z_{i-1})=-f'(z_i)\neq 0$. In fact, in terms of the original angles the periodic identifications imply $(\tilde{\phi}^1, \tilde{\phi}^2)\sim (
\tilde{\phi}^1,\tilde{\phi}^2+ \ell_i)$  and hence these must all be multiples of some basic period $(\tilde{\phi}^1, \tilde{\phi}^2)\sim (
\tilde{\phi}^1,\tilde{\phi}^2+ \ell)$ so in fact $\xi_i\sim \xi_i +\ell$ and hence $\ell_i=\ell$ for all $i$ after all.  Therefore,  near any endpoint of each $I_i$, 
\[
H_{\rho=0} = \frac{\ell}{4\pi |z-z_i|} + O(1)
\]
as $z\to z_i$ and since $H$ is an axisymmetric harmonic function on $\mathbb{R}^3$ with isolated singularities at $z=z_i$ this implies $H$ must have a simple pole at each corner $z_i$. Thus we can write
\[
H= H_0+ \frac{\ell}{4\pi} \sum_{i=1}^n \frac{1}{  \sqrt{ \rho^2+ (z-z_i)^2}}
\]
for some harmonic function $H_0$ that is smooth everywhere on $\mathbb{R}^3$.  Finally, both ALE and ALF asymptotics imply that $|\partial_{\tilde\phi^2}|^2$  grows as $r^2$ or approaches a constant, which in turn implies $H$ is either $O(r^{-2})$ or $O(1)$.  Thus in either case $H_0$ is bounded and hence must be a constant. 

It is convenient to rescale the angle $\phi^2= \tilde{\phi}^2 (4\pi/\ell)$ so it is $4\pi$ periodic and by scaling $H$ and $\chi$ by a factor of $4\pi /\ell$, we can rescale the metric by the same overall factor to effectively set $\ell=4\pi$.  Hence without loss of generality we have
\[
H = c + \sum_{i=1}^n \frac{1}{  \sqrt{ \rho^2+ (z-z_i)^2}}
\]
where $c$ is a constant and integrating for $\chi$ we get
\[
\chi = \sum_{i=1}^n \frac{(z-z_i)}{  \sqrt{ \rho^2+ (z-z_i)^2}}
\]
where by shifting $\phi^2$ appropriately we have set an additive constant to zero.
If  $c=0$ vanishes the solution is ALE and is the multi-Eguchi-Hanson, whereas if $c>0$ the solution is ALF and it is the multi-Taub-NUT.
\end{proof}

We can extract the rod structure of the self-dual solutions as follows. From the above explicit solution $\chi_i = 2i-2-n$ for $i=1, \dots n+1$ and therefore we have the rod vectors $v_i=\partial_{\phi^1}- \chi_i \partial_{\phi^2}$ explicitly. If we choose $v_1$ and $v_2$ as a basis, then we can write a general rod vector as
\[
v_i = (2-i) v_1+ (i-1) v_2
\]
for $i=1, \dots, n+1$.  Therefore, choosing a basis $\underline{v}_1=(0,-1)$ and $\underline{v}_2=(1,0)$  we find $\underline{v}_i = (i-1, i-2)$ and in particular $\underline{v}_{n+1}= (n, n-1)$ so the cross-section as infinity is the lens space $L(n,1)$.

\appendix
\section{Alternate ALE model map}

In this appendix we will give an alternate model map for the ALE case. Thus first consider the flat asymptotic cone metric \eqref{eq:ALEflat}.
Weyl coordinates adapted to the Killing fields $\partial_{\bar{\phi_i}}$ are given by 
\[
\rho = \tfrac{1}{2} \bar r^2\sin 2\bar\theta, \qquad z= \tfrac{1}{2}\bar r^2 \cos 2\bar\theta
\]
in terms of which the metric is
\[
\bar{\mathbf{g}}= \frac{1}{2 \sqrt{\rho^2+z^2}} ( \td \rho^2+ \td z^2) + \mu^+ \td \bar{\phi}_1^2+ \mu^- \td \bar{\phi}_2^2
\]
where $\mu^\pm = \sqrt{\rho^2+ z^2} \mp z $ and $\bar r^2 = 2 \sqrt{\rho^2+z^2}$.

Now, introduce a basis defined by independently $2\pi$-periodic angles $\phi_1, \phi_2$, say $\bar{\phi}_1= p^{-1} \phi_1$  and $\bar{\phi}_2 = \phi_2 +   q p^{-1} \phi_1$. Then, in the basis $\partial_{\phi_i}$ the Gram matrix is 
\begin{equation}
g = L D_0  L^T , \qquad D_0= \begin{pmatrix}  \mu^+ & 0 \\ 0 & \mu^-
\end{pmatrix}, \qquad  L = \begin{pmatrix} p^{-1} & q p^{-1} \\ 0 & 1 \end{pmatrix} \; ,
\end{equation}
so in particular $\det g = p^{-2} \rho^2$ and 
\[
g|_{\rho=0, z<0} \begin{pmatrix}
    0 \\ 1
\end{pmatrix} = 0 , \qquad g|_{\rho=0, z>0} \begin{pmatrix}
    p \\ -q
\end{pmatrix} = 0  \; .
\]
Therefore, the two semi-infinite rods have rod vectors corresponding to a lens space $L(p,q)$.  

We are now ready to define the model map. Following the same setup and notation as in the main text, fix $z_1=0$ so $z_n>0$, then on $C_R, \mathcal{S}_1, \mathcal{S}_{n+2}$ we take the model map to be given by the above so the rod vectors on $I_1$ and $I_{n+1}$ are  $\underline{v}_1=(0,1)$ and $\underline{v}_{n+1}=(p,-q)$. On the regions $\mathcal{S}_{i=2, \dots, n+1}$ and the transition regions $\mathcal{T}_{i=2, \dots, n}$ we take the same model map as in the main text.

On the transition region $\mathcal{T}_1$, which separates $\mathcal{S}_1$ and $\mathcal{S}_2$,  we take 
\begin{equation}
    g = N_1(z) D_0^{1-\lambda(z)}  D_1^{\lambda(z)}  N_1(z)^T  , \qquad D_1= \begin{pmatrix}
        \mu_1^+ & 0 \\ 0 & \mu_1^-
    \end{pmatrix}\; ,
\quad 
N_1(z) = \begin{pmatrix}
    p^{-1} & \beta (z) \\ 0 & 1 
\end{pmatrix}
\end{equation}
where $\lambda(z), \beta(z)$ are smooth functions such that $0\leq \lambda(z) \leq 1$ and\footnote{Since we have fixed $z_1=0$ in fact $D_0=D_1$ and hence the $\lambda$ dependence drops out. We have chosen to write it in this way for clarity.}
\[
\lambda(z) = \begin{cases}
    0 \\1 
\end{cases}
\qquad \beta(z) = \begin{cases}
    qp^{-1} \\ 
    \epsilon_1 v^2_2
\end{cases}
\qquad \qquad \text{if} \qquad \begin{array}{c} z\in S_1 \\ z\in S_2 \end{array}  \; .
\]
Thus $N_1$ interpolates from $L$ to 
\[
M_1 = \begin{pmatrix}
    p^{-1} & \epsilon_1 v_2^2 \\ 0 & 1 
\end{pmatrix}
\]
where the latter is defined in \eqref{eq:modmapinterior} (and recall $v_1=(0,1)$ so by the admissibility condition $\epsilon_1 v_2^1=-1$). Hence $g$ interpolates from its value on $\mathcal{S}_1$ and $\mathcal{S}_2$, and furthermore,  $\det g = p^{-2} \rho^2$, $N_1^T \underline{v}_1 = \underline{v}_1$ so $v_1\in \text{ker}(g)_{\rho=0}$ on the transition region, as required.

On the transition region $\mathcal{T}_{n+1}$, which separates $\mathcal{S}_{n+1}$ and $\mathcal{S}_{n+2}$, we take
\begin{equation}
    g = N_{n+1}(z) D_0^{\lambda(z)} D_n^{1-\lambda(z)} N_{n+1}(z)^T, \qquad D_n= \begin{pmatrix}
        \mu_n^+ & 0 \\ 0 & \mu_n^-
    \end{pmatrix} \; ,
\qquad N_{n+1}(z) = \begin{pmatrix}
        \frac{1}{pA}+ \frac{q C}{p} &\frac{q A}{p}  \\ C & A
    \end{pmatrix}
\end{equation}
where $\lambda(z), A(z), C(z)$ are smooth functions such that $0\leq 
\lambda(z)\leq 1$, $A(z)>0$ and
\[
\lambda(z) = \begin{cases}
    0 \\1 
\end{cases}
\qquad 
A(z) = \begin{cases}
    -\epsilon_n p \\ 1
\end{cases}
\qquad
C(z)= \begin{cases}
   -p^{-1} v_n^1  \\ 0 
\end{cases}
\qquad \qquad \text{if} \qquad \begin{array}{c} z\in S_{n+1} \\ z\in S_{n+2} \end{array}  \; .
\]
Thus $\det N_{n+1}=p^{-1}$ and $N_{n+1}$ smoothly interpolates between 
\[
M_n = \begin{pmatrix}
    p^{-1} v_n^2 & -\epsilon_n q \\ - p^{-1} v_n^1 & -\epsilon_n p
\end{pmatrix}
\]
which is defined in \eqref{eq:modmapinterior} (recall  $\underline{v}_{n+1}= (p, -q)$ and note that by the admissibility condition $\det M_n=1/p$) and $L$.  The requirement that $A>0$ implies we must fix $\epsilon_n=-1$.  Hence $g$ smoothly interpolates between its values on $\mathcal{S}_{n+1}$ and $\mathcal{S}_{n+2}$. Furthermore, note that $\det g= p^{-2} \rho^2$, and 
\[
N_{n+1}^T \begin{pmatrix}
    p \\ - q
\end{pmatrix} =A^{-1} \begin{pmatrix}
   1 \\ 0
\end{pmatrix}
\]
so $v_{n+1}\in \text{ker} g|_{\rho=0}$ on the whole transition region as required.  

Finally, we must check the tension on each of these transition regions. It is a straightforward calculation to check that as  $\rho\to 0$ on $\mathcal{T}_1$  one has
\[
E^1_{~1}=-E^2_{~2}= E^1_{~2}=O(\rho^2), \qquad E^2_{~1}= O(1) \; ,
\]
whereas on $\mathcal{T}_{n+1}$ one has $E^i_{~j}= O(1)$. Therefore, the norm of the tension $|\tau|^2 = \text{Tr} E^2$ is indeed bounded in these regions.


\begin{thebibliography}{99}

 
\bibitem{Aksteiner:2021fae}
S.~Aksteiner and L.~Andersson,
{\it Gravitational instantons and special geometry},
J. Diff. Geom. \textbf{128} (2024) no.3, 928-958.

\bibitem{Aksteiner:2023djq}
S.~Aksteiner, L.~Andersson, M.~Dahl, G.~Nilsson and a.~Simon,
{\it Gravitational instantons with $S^1$ symmetry},
[arXiv:2306.14567 [math.DG]].


\bibitem{Araneda:2025uqo}
B.~Araneda and J.~Lucietti,
{\it All toric Hermitian ALE gravitational instantons},
[arXiv:2510.09291 [math.DG]].



\bibitem{BeigI}
R. Beig, {\it The static gravitational field near spatial infinity I}. Gen. Rel. Grav. \textbf{12}, 439–451 (1980). 

\bibitem{BeigII}
R.~Beig and W.~Simon, {\it The stationary gravitational field near spatial infinity},
Gen. Rel. Grav. \textbf{12}, 1003-1013 (1980)



\bibitem{Biquard:2021gwj}
O.~Biquard and P.~Gauduchon,
{\it On Toric Hermitian ALF Gravitational Instantons,}
Commun. Math. Phys. \textbf{399} (2023) no.1, 389-422.


\bibitem{BKN}
S.~Bando, A.~Kasue, and H.~Nakajima, 
{\it On a construction of coordinates at infinity on manifolds with fast curvature decay and maximal volume growth}, Inventiones mathematicae, 97(2), 313-349 (1989).


\bibitem{ChenChen}
G. Chen and X. Chen, {\it Gravitational instantons with faster than quadratic curvature decay (I)}, 
Acta Math., 227 (2021), 263–307.

\bibitem{Chen:2011tc}
Y.~Chen and E.~Teo,
{\it A New AF gravitational instanton},
Phys. Lett. B \textbf{703} (2011), 359-362.

\bibitem{Chen:2010zu}
Y.~Chen and E.~Teo,
{\it Rod-structure classification of gravitational instantons with} $U(1) \times U(1)$ {\it isometry},
Nucl. Phys. B \textbf{838} (2010), 207-237.



\bibitem{Gibbons}
G.~W.~Gibbons, 
{\it Gravitational instantons: A survey},  in: Osterwalder, K. (eds) Mathematical Problems in Theoretical Physics: Lecture Notes in Physics, vol 116. Springer, Berlin, Heidelberg .(1980).

\bibitem{Han:2022mmc}
Q.~Han, M.~Khuri, G.~Weinstein and J.~Xiong,
{\it Asymptotic Analysis of Harmonic Maps With Prescribed Singularities,}
[arXiv:2212.14826 [math.DG]].


\bibitem{Hawking:1976jb}
S.~W.~Hawking,
{\it Gravitational Instantons,}
Phys. Lett. A \textbf{60} (1977), 81.


\bibitem{Hollands:2007aj}
S.~Hollands and S.~Yazadjiev,
{\it Uniqueness theorem for 5-dimensional black holes with two axial Killing fields},
Commun. Math. Phys. \textbf{283}, 749-768 (2008).

\bibitem{Hollands:2008fm}
S.~Hollands and S.~Yazadjiev,
{\it A Uniqueness theorem for stationary Kaluza-Klein black holes,}
Commun. Math. Phys. \textbf{302} (2011), 631-674


\bibitem{Khuri:2022xdy}
M.~Khuri, M.~Reiris, G.~Weinstein and S.~Yamada,
{\it Gravitational solitons and complete Ricci flat Riemannian manifolds of infinite topological type},
Pure Appl. Math. Quart. \textbf{20} (2024) no.4, 1895-1921.


\bibitem{Kronheimer:1989zs}
P.~B.~Kronheimer,
{\it The construction of ALE spaces as hyper-K{\"a}hlerquotients},
J. Diff. Geom. \textbf{29} (1989) no.3, 665-683.

\bibitem{Kronheimer:1989pu}
P.~B.~Kronheimer,
{\it A Torelli type theorem for gravitational instantons},
J. Diff. Geom. \textbf{29} (1989) no.3, 685-697.

\bibitem{Kunduri:2021xiv}
H.~K.~Kunduri and J.~Lucietti,
{\it Existence and uniqueness of asymptotically flat toric gravitational instantons},
Lett. Math. Phys. \textbf{111} (2021) no.5, 133.



\bibitem{Lapedes:1980st}
A.~S.~Lapedes,
{\it Black Hole Uniqueness Theorems in Euclidean Quantum Gravity}, Phys. Rev. D \textbf{22} (1980), 1837.

\bibitem{LiALF}
Li, M., 2023. {\it Classification results for conformally K\" ahler gravitational instantons}, arXiv:2310.13197 [math.DG].

\bibitem{Li2023} 
M.~Li, 
{\it On 4-dimensional Ricci-flat ALE manifolds},
arXiv:2304.01609 [math.DG].


\bibitem{LiSun} 
M.~Li and S.~Sun, 
{\it Gravitational instantons and harmonic maps}, 
arXiv:2507.15284 [math.DG].


\bibitem{Mars:1998epf}
M.~Mars and W.~Simon,
{\it A proof of uniqueness of the Taub-bolt instanton},
J. Geom. Phys. \textbf{32} (1999), 211-226.

\bibitem{Minerbe:2010yrr}
V.~Minerbe,
Annales Sci. Ecole Norm. Sup. \textbf{43} (2010) no.6, 883-924
doi:10.24033/asens.2135

\bibitem{Minerbe} V.~Minerbe, {\it A mass for ALF manifolds} Comm. Math. Phys. \textbf{289},  (2009) 925--955. 



\bibitem{Minerbe:2009gj}
V.~Minerbe,
{\it Rigidity for Multi-Taub-NUT metrics}, J. Reine Angew. Math. {\bf 656} (2011), 47--58.


\bibitem{Nakajima}
H. Nakajima, 
{\it Self-duality of ALE Ricci-flat 4-manifolds and Positive Mass Theorem}, Advanced Studies in Pure Mathematics 18-1, 1990, pp. 385-396.

 
\bibitem{Nilsson:2023ina}
G.~Nilsson,
{\it Topology of toric gravitational instantons},
Differ. Geom. Appl. \textbf{96} (2024), 102171.

\bibitem{OR}
Orlik, P. and Raymond, F. {\it Actions of the torus on 4-manifolds I}, Transactions of the AMS 152, (1972).


\bibitem{SunZhang}
S. Sun and R. Zhang, {\it Collapsing geometry of hyperkähler 4-manifolds and applications}, Acta Mathematica 2024

\bibitem{wein}
G. Weinstein, {\it On rotating black holes in equilibrium in general relativity}, Commun. Pure Appl. Math. 43 (1990) 903-948 

\bibitem{Weinstein:2019zrh}
G.~Weinstein,
{\it Harmonic maps with prescribed singularities and applications in general relativity},
[arXiv:1902.01576 [gr-qc]].



\end{thebibliography}
\end{document}